\documentclass[reqno,11pt]{amsart}
\usepackage{tkz-euclide}
\usepackage{hyperref}
\usepackage{amsmath}  
\usepackage{amsfonts}
\usepackage{amssymb} 
\usepackage{mathrsfs} 
\usepackage{graphicx}  
\usepackage{bm}
\usepackage{tikz} 
\usetikzlibrary{svg.path} 

\usetikzlibrary{calc} 
\usepackage{xcolor}  
\usetikzlibrary{patterns}
\usetikzlibrary{arrows.meta} 
\usepackage{amsthm} 
\usepackage{float}
\usepackage{enumitem}
\usepackage[english]{babel}
\usepackage[font=footnotesize, labelfont=bf]{ caption}
\usepackage{ esint }
\usepackage{stackengine,scalerel,graphicx}
\stackMath

\definecolor{trueblue}{rgb}{0.0, 0.45, 0.81}
\definecolor{truegreen}{rgb}{0.13, 0.55, 0.13}

\newcommand{\eps}{\varepsilon} 
\newcommand{\dx}{\, {\rm d}x}
\newcommand{\e}{\varepsilon}

\theoremstyle{plain}
\begingroup
\newtheorem{theorem}{Theorem}[section]
\newtheorem{lemma}[theorem]{Lemma}

\newtheorem{remark}[theorem]{Remark}
\newtheorem{proposition}[theorem]{Proposition}
\newtheorem{corollary}[theorem]{Corollary}
\newtheorem*{ltheorem}{Liouville's Theorem}
\endgroup

\newenvironment{step}[1]{\textbf{Step #1}.}{}

\theoremstyle{definition}
\begingroup

\endgroup

\renewcommand{\tilde}{\widetilde}

\renewcommand{\d}{ \mathrm{d}}

\DeclareMathOperator{\dv}{div}

\DeclareMathOperator{\dist}{dist}

\numberwithin{equation}{section}
\newcommand{\N}{\mathbb{N}}

\newcommand{\R}{\mathbb{R}}

\renewcommand{\S}{\mathbb{S}}

\renewcommand{\H}{\mathcal{H}}

\usepackage[normalem]{ulem}

\def \mb{\mathbb}
\def \n{\nabla}
\def \d{\mathrm{d}}
\def \mc{\mathcal}
\def \Mob{\textup{M\"ob}_+}
\def \Mobr{\textup{M\"ob}_-}
\def \SO{\textup{SO}}
\newcommand{\weak}{\rightharpoonup}

\begin{document}
	
\title[Sharp stability of the M\"obius group in arbitrary dimension]{Sharp quantitative stability of the M\"obius group\\ among sphere-valued maps in arbitrary dimension}
	
\author[Andr\'e Guerra]{Andr\'e Guerra} 
\address[Andr\'e Guerra]{Institute for Theoretical Studies, ETH Zürich, CLV, Clausiusstrasse 47, 8006 Zürich, Switzerland}
\email{andre.guerra@eth-its.ethz.ch}

\author[Xavier Lamy]{Xavier Lamy} 
\address[{Xavier Lamy}]{Institut de Mathématiques de Toulouse, Université Paul Sabatier, 118, route de Narbonne, F-31062 Toulouse Cedex 8, France}
\email{Xavier.Lamy@math.univ-toulouse.fr}	
	
\author{Konstantinos Zemas}
\address[Konstantinos Zemas]{Institute for Analysis and Numerics, University of M\"unster\\
Einsteinstrasse 62, 48149 M\"unster, Germany}
\email{konstantinos.zemas@uni-muenster.de}

\begin{abstract}
\footnotesize
In this work we prove a sharp quantitative form of Liouville's theorem, which asserts that, for all $n\geq 3$, the weakly conformal maps of $\mb S^{n-1}$ with degree  $\pm 1$ 
are  M\"obius transformations.
In the case $n=3$ this estimate was first obtained by Bernand-Mantel, Muratov and Simon 
(Arch.\ Ration.\ Mech.\ Anal.\ 239(1):219-299, 2021), 
 with  different proofs given later on by Topping, and by Hirsch and the third author.
The higher-dimensional case $n\geq 4$ requires new arguments because it is genuinely nonlinear: the linearized version of the estimate involves quantities which cannot control the 
distance to M\"obius transformations  in the conformally invariant Sobolev norm.
Our main tool to circumvent this difficulty is an inequality introduced by Figalli and Zhang in their proof of a sharp stability estimate for the Sobolev inequality.

\end{abstract} 
	
\maketitle

	
\section{Introduction}\label{sec:1}

Rigidity theorems of geometric nature are ubiquitous in analysis, geometry and mathematical physics, one of the most classical, yet prominent examples being \textit{Liouville's theorem} on the characterization of isometries and conformal transformations among domains of $\R^n$ ($n\geq 2$, resp.\ $n\geq 3$). Its qualitative as well as quantitative extensions (cf.\ for instance 
{\cite{faraco2005geometric, friesecke2002theorem,Iwaniec2002,mullersverakyan,reshetnyak1994stability, Yan_1, Yan_2} and the references therein) have formed a very active research direction over the recent years, not only because of a pure geometric interest, but also in view of applications in mathematical models of materials science.


Our starting point is the spherical version of Liouville's theorem for conformal maps.  A Sobolev map $u\in W^{1,n-1}( \mb S^{n-1};\mb S^{n-1})$ is called \textit{weakly conformal} if and only if 
\begin{equation}
\label{eq:genconf}
(\n_T u)^{t} \n_T u= \frac{|\n_T u|^2}{n-1} I_{ x}
\end{equation}
for  $\H^{n-1}$\text{-a.e.} $x$ on $\S^{n-1}$.  
Here and throughout the paper $\nabla_T u\in \R^{n\times (n-1) 
}$ denotes the extrinsic gradient of $u$  with respect to a 
local orthonormal frame $\{\tau_1,\dots, \tau_{n-1}\}$ on $\S^{n-1}$  indicated by the unit normal, and $I_x$ stands for the identity transformation on $T_x\S^{n-1}$. For maps in that class one  can define the \textit{topological degree} of $u$ through
\begin{equation}\label{def: degree}
\deg u := \fint_{\S^{n-1}}\bigg\langle u,\bigwedge_{i=1}^{n-1}\partial_{\tau_i}u\bigg\rangle\,\mathrm{d}\H^{n-1}\,,
\end{equation}
 see e.g.\ \cite{brezis1995degree} for more information on the degree. 

Let us write $\Mob(\S^{n-1})$ for the group of \textit{orientation-preserving Möbius maps} of $\S^{n-1}$, see  Subsection \ref{subs:conformal_maps}  
for the precise definition.  It is a remarkable geometric fact that solutions of \eqref{eq:genconf} are very rigid:

\clearpage
\begin{ltheorem}
Let $u\in W^{1,n-1}(\mb S^{n-1};\mb S^{n-1})$ be a solution of \eqref{eq:genconf}.
\begin{enumerate}[ref={\normalfont(\roman*)},label={\normalfont(\roman*)}]
\item if $n=3$ and $\deg u =1$, then $u\in \Mob(\mb S^{n-1})$;
\smallskip
\item\label{it:n>3} if $n\geq 4$ and $\deg u \geq 1$, then $u\in \Mob(\mb S^{n-1})$\,.
\end{enumerate}
\end{ltheorem}
In fact, \ref{it:n>3} holds locally even if $u$  solves  \eqref{eq:genconf} only on a domain of $\mb S^{n-1}$  \cite{Iwaniec2002}, while for $n=3$ 
 solutions to \eqref{eq:genconf} of general degree  have also been classified  \cite{Lemaire, wood}.  We refer the reader to  \cite[Appendix A]{zemas2022rigidity} for a simple proof of Liouville's Theorem when $\deg u = 1$.

In this paper we prove a sharp quantitative version of Liouville's Theorem  on $\S^{n-1}$  in all dimensions $n\geq 3$.  To state the result, we consider the class of admissible maps
\begin{equation}\label{admissible_class_of_maps}
\mathcal{A}_{\S^{n-1}}:=\left\{u\in W^{1,n-1}(\S^{n-1};\S^{n-1}):\  \mathrm{deg}\, u=1\right\}\,,
\end{equation}
and the deficit 
\begin{equation}\label{deficit_on_S_n-1}
\delta_{n-1}(u):=\fint_{\S^{n-1}}\bigg(\frac{|\nabla_T u|^2}{n-1}\bigg)^{\frac{n-1}{2}}\, \mathrm{d}\mathcal{H}^{n-1}-1\,.
\end{equation}
Note that the deficit  in \eqref{deficit_on_S_n-1}  is clearly a  conformally-invariant quantity which, by  a well-known inequality relating the $(n-1)$-Dirichlet energy and the degree of maps in $W^{1,n-1}(\S^{n-1};\S^{n-1})$,  
satisfies $\delta_{n-1}(u) \geq 0$ for all $u\in \mc A_{\mb S^{n-1}}$, with 
$$ \delta_{n-1}(u)=0 \iff  u\in \Mob(\mb S^{n-1})\,,$$
see  Corollary 
 \ref{cor:deficit} below.

Our main result shows that $\delta_{n-1}(u)$ measures  in a sharp fashion the distance of $u$ from a particular M\"obius map in an integral sense, namely we have the following.  
 
\smallskip

\begin{theorem}\label{main_thm}
Let $n\geq 3$. There exists a constant $C_n>0$ such that for every $u\in \mathcal{A}_{\S^{n-1}},$ 
\begin{equation}\label{conf_S_S_quantitative}
\inf_{\phi \in \Mob(\mb S^{n-1})} \fint_{\S^{n-1}} \left|\nabla_T u-\nabla_T \phi\right|^{n-1} \d \mc H^{n-1} \leq C_n \delta_{n-1}(u)\,.
\end{equation}
\end{theorem}

\smallskip

\begin{remark}[\text{Sharpness}]\label{sharpness}
\normalfont Estimate \eqref{conf_S_S_quantitative} is sharp in the sense that, on the right-hand side,  the deficit  cannot be replaced with $\delta_{n-1}(u)^\beta$ for some $\beta>1$.  An example addressing the optimality of the exponent is discussed in Appendix \ref{optimality_example}.
\end{remark}

\begin{remark}[Maps with other degrees]
\normalfont By  inequality \eqref{ineq: wente_on_sphere},  we have
$$| \deg u | >1 \quad \implies \quad \delta_{n-1}(u) \geq |\deg u|-1\geq 1\,.$$
In particular,  this easily implies that \eqref{conf_S_S_quantitative} also holds for $u\in W^{1,n-1}(\mb S^{n-1};\mb S^{n-1})$ with $|\deg u | >1$, cf.\ Step 0 in Section \ref{sec:proof} below.
Likewise,  the case of maps of degree $-1$ is completely analogous to the degree 1 case, and can simply be derived by the latter, by composing any map of degree -1 with the orientation-reversing  orthogonal map $(x_1,\dots,x_{n-1}, x_n)\mapsto (x_1,\dots,x_{n-1},-x_n)$, and replacing $\Mob(\mb S^{n-1})$ with $ \Mobr(\mb \S^{n-1})$ (the group of orientation-reversing M\"obius maps). 
\end{remark}

Theorem \ref{main_thm} was first proved for $n=3$ in \cite[Theorem 2.4]{bernand2019quantitative}, with simpler and different proofs having been supplied in \cite{hirsch2022mobius2, topping2020}.  
Thus, the novelty of this paper is the higher dimensional, genuinely non-linear case $n\geq 4$. 
We also refer to \cite{ Deng_et_al_2,Deng_et_al_1} for related results for half-harmonic maps from $\R$ to $\S$ and a local stability result for maps from $\S^2$ to $\S^2$ with higher degree, 
 where the assumption of apriori closeness to a harmonic map however prevents bubbling phenomena for almost energy-minimizing maps. Without this strong assumption, an optimal quantitative improvement has been established very recently in \cite{rupflin}, which, rougly speaking, asserts that maps from $\S^2$ to $\S^2$ of higher degree and small excess energy are close to a collection of rational maps that
describe the behaviour at very different scales.  
The interested reader is also referred to \cite{Engelstein_2022_Neumayer_Spolaor} and the references therein for further quantitative stability results related to problems in conformal geometry.

Our proof  of Theorem \ref{main_thm}  is based on the approach from \cite{hirsch2022mobius2},  which however does not extend trivially to higher dimensions, the main difficulty being the  genuine nonlinearity and  degenerate convexity of the conformal Dirichlet energy. 

 \subsection*{Sketch of proof.} Let us give here a sketch of the argument,  which can be thought of as a quantitative perturbation of the one in \cite[Appendix A]{zemas2022rigidity} for the exact case. The starting point of the proof is that it suffices to prove Theorem \ref{main_thm} for maps $u\in \mathcal{A}_{\S^{n-1}}$ for which
\[\delta_{n-1}(u)\ll 1\,, \qquad \fint_{\S^{n-1}}u=0\,, \qquad \|u-x\|_{W^{1,n-1}(\S^{n-1})}\ll 1\,.\]
While the first reduction is clear, the reduction to zero-mean maps can always be achieved by precomposition with a Möbius transformation,  see  Lemma \ref{lemma:topological}, while the last one is ensured by a contradiction/compactness
 argument, see Lemma \ref{lemma:compactness} and Step 3 in Section \ref{sec:proof}. After these reductions, and since $u$ is $\S^{n-1}$-valued, in the case $n=3$ of \cite{hirsch2022mobius2}, the deficit transforms into the deficit in the sharp Poincar{\' e} inequality on $\S^2$ (cf. \eqref{eq:firstpoincare}, \eqref{eq:secondpoincare}), namely
\[\delta_2(u)=\frac{1}{2}\fint_{\S^2}|\nabla_Tu|^2-\fint_{\S^2}|u|^2\gtrsim \fint_{\S^2}|\nabla_T (u-Ax)|^2\,,\]
where $Ax$ is the linear part of $u$, i.e., its projection onto the first order spherical harmonics (cf. Subsection \ref{subsec:spherical}). Keeping in mind that $u$ was precomposed with a Möbius map to achieve the zero-mean condition, the proof is concluded in this case by showing that $A$ can be replaced in the above estimate by a rotation matrix, i.e., that 
\[\mathrm{dist}^2(A;\SO(n))\lesssim \delta_2(u)\,.\]
The last estimate is achieved by an appropriate expansion of the formula for the degree \eqref{def: degree} around the linear part of $u$. 

In the case $n\geq 4$, which is the case of study of this work, the previous argument however cannot be adapted directly. Indeed, setting (for simplicity) $w:=u-x$, and naively Taylor-expanding the deficit around the identity would formally yield
\begin{align*}
\delta_{n-1}(u) & =\frac{1}{2}\bigg(\fint_{\S^{n-1}}|\nabla_T w|^2-(n-1)\fint_{\S^{n-1}}|w|^2+\frac{n-3}{n-1}\fint_{\S^{n-1}}(\mathrm{div}_{\S^{n-1}}w)^2\bigg)\\
& \qquad +\fint_{\S^{n-1}}\mathcal{O}(|\nabla_Tw|^3)\,,
\end{align*}
which however does not give a control on the desired optimal norm we want in the left hand side of \eqref{conf_S_S_quantitative}. Moreover, even though the first part of the quadratic form appearing in the expansion gives again the deficit in the sharp Poincar{\' e} inequality on $\S^{n-1}$, a suboptimal estimate with respect to the $L^2$-norm  on the left hand side of \eqref{conf_S_S_quantitative} would not apriori be possible, because of the presence of the higher-order (cubic) terms in the expansion.   

To overcome these issues we consider a  Taylor-type lower inequality instead of the exact  expansion for the conformal Dirichlet energy (cf.   Lemma~\ref{lemma:figalli} below),  which can be thought of as an appropriate interpolation between the  first up to quadratic  terms, and the desired $\fint_{\S^{n-1}}|\nabla_T w|^{n-1}$-term. Such an expansion is a special case of the ones  introduced in \cite{figalli2022sobolev} to prove a sharp quantitative version of the Sobolev inequality  in $\R^n$, cf.\ also \cite{bianchi1991egnell, figalli2010neumayer, frank} for related results in this direction. Surprisingly, such an expansion finally leads us to an estimate of the form
\[\fint_{\S^{n-1}}|\nabla_T (u-Ax)|^{n-1}\lesssim \delta_{n-1}(u)+\mathrm{dist}^2(A;\SO(n))\,,\]
where again $Ax$ is the linear part of $u$. The proof would again be concluded if the last term on the right hand side of the above inequality is optimally estimated by the deficit itself, which is achieved also by an appropriate expansion of the degree of $u$ around its linear part, and $L^2$-$L^{n-1}$ interpolation estimates, see Step (2b) in Section \ref{sec:proof}.  

\begin{remark}\label{toppings_approach}\normalfont Interestingly, the observation that the group $\Mob(\S^{n-1})$  can be used to fix the mean
value of maps in $\mathcal{A}_{\S^{n-1}}$ to 0 was also used in the proof of Theorem \ref{main_thm} in the case $n=3$ by P. Topping \cite{topping2020}. While in \cite{hirsch2022mobius2} and here we
used it to link the problem to the stability of the sharp Poincar{\' e} inequality on $\S^{n-1}$, in \cite{topping2020} it is used in order to start the harmonic map heat flow (cf.\ \cite {LinWang}) with initial datum $u \in \mathcal{A}_{\S^2}$ with $\fint_{\S^2} u=0$. With
this centering, the flow does not produce a bubble in finite time and converges as $t\to\infty$ to a map in $\Mob(\S^{2})$. This limiting map turns out to be the one for which the desired stability
estimate is satisfied. It would be interesting to see if  such an approach can be followed to get an alternative proof of  Theorem \ref{main_thm} also in the case $n\geq 4$ by using the nonlinear $(n-1)$-harmonic map heat flow 
\cite{hungerbuhler97},
 which will be a question for future work.\\[-18pt]
\end{remark}

 \subsection*{Plan of the paper.} The plan of the paper is the following. After fixing some notation, in Section \ref{sec:geometry} we collect some basic facts about spherical harmonics and Möbius maps on $\S^{n-1}$, as well as some boundary expressions of integrals of Jacobian subdeterminants (null-Lagrangians) over the unit ball. These will be useful in the proof of Theorem \ref{main_thm}, which is the content of Section 3. For the reader's convenience the proof is divided into several steps, following in spirit and method the approach of \cite{hirsch2022mobius2} for the case $n=3$. Step 2 therein (divided into further substeps) contains the main novelty in the argumentation for the higher dimensional case $n\geq 4$, as these were sketched previously. 
  In Appendix \ref{optimality_example} we present an example showing the optimality of the estimate \eqref{conf_S_S_quantitative} (cf. Remark \ref{sharpness}). 

\subsection*{Notation}
We use $\langle\cdot,\cdot\rangle$ to denote the inner product between vectors in $\R^n$ and $A:B$ to denote the Hilbert--Schmidt inner product between any two matrices $A,B\in \R^{n\times m}$.  
For every two vectors $a\in \R^n$ and $b\in \R^m$ we denote by $a\otimes b\in \R^{n\times m}$ their tensor product, 
\[(a\otimes b)_{ij}:=a_ib_j\ \ \forall i=1,\dots,n \ \text{and } j=1,\dots,m\,.\]
As mentioned above, $\nabla_Tu$ denotes the extrinsic gradient of $u\colon \mb S^{n-1}\to \mb S^{n-1}$ (when viewed as a map with values in $\R^n$),   represented in local coordinates by the $n\times (n-1)$-matrix with entries \[(\nabla_T u)_{ij} = \langle \nabla_T u^i, \tau_j\rangle\,,\quad  \forall i=1,\dots,n\, \ \mathrm{and }\ j=1,\dots,n-1\,.\]
 Here, $\{\tau_1,\dots,\tau_{n-1}\}$ is a local orthonormal frame on $T_x\S^{n-1}$ indicated by the unit normal, i.e., for every $x\in \S^{n-1}$ the set of vectors $\{\tau_1(x),\dots, \tau_{n-1}(x), x\}$ is a positively oriented orthonormal frame of $\R^n$. 

By $\mathrm{id}_{\S^{n-1}}, I_x, I_n$ we denote the identity map on $\S^{n-1}$, the identity transformation on $T_x\S^{n-1}$, and the identity matrix in $\R^{n\times n}$, respectively. We also set for brevity,
\[P_T:=\nabla_T\mathrm{id}_{\S^{n-1}}\,,\ \mathrm{i.e.}\,,\ (P_T)_{ij}=\langle e_i,\tau_j\rangle\,.\]
 By $\omega_n$ we denote the Euclidean volume of the unit ball in $\R^n$ and by $\SO(n)$ we label as usual the group of orientation-preserving rotation matrices, i.e.,
\[\SO(n):=\{R\in \R^{n\times n}\colon R^tR=I_n\, \ \text{and } \mathrm{det}R=1\}\,.\]
For any $A\subset \R^n$, we denote by $\mathbf{1}_A$ its 
indicator
 function, i.e., 
\begin{align*}
\mathbf{1}_A(y):=\begin{cases} 1\,, \quad y\in A\,,\\
	0\,, \quad y\notin A\,.
\end{cases}
\end{align*}
We consider maps in the Sobolev space 
\[W^{1,n-1}(\S^{n-1};\S^{n-1}):=\{u\in W^{1,n-1}(\S^{n-1} ; \R^n)\colon \ |u(x)|=1 \ \text{for } \H^{n-1}\text{-a.e. } x\in \S^{n-1}\}\,.\]
 Note that, in the definition of the degree for such maps in \eqref{def: degree}, we identify $\bigwedge_{i=1}^{n-1}\partial_{\tau_i}u$ with $*\bigwedge_{i=1}^{n-1}\partial_{\tau_i}u$ through Hodge-duality. 
 
By $ \lesssim,  \ll $ we mean that the corresponding inequality is valid up to a
 multiplicative (resp. sufficiently small) 
 constant that depends only on the dimension, but is allowed to vary from line to line. 
  With this notation, for any two quantities $F,G$ we say that $F\sim G$ iff $F\lesssim G$ and $G\lesssim F$.

\section{Geometry of $\mathbb S^{n-1}$  and null-Lagrangians}
\label{sec:geometry}

\subsection{Spherical harmonics}\label{subsec:spherical} 
We begin this section by recalling some useful facts about spherical harmonics, referring the reader to \cite{groemer} for further details.
For every $k\in\mathbb{N}$ let $H_{n,k}$ be the subspace of $L^2(\S^{n-1};\R^n)$ consisting of vector fields whose components are $k$-th order \textit{vectorial spherical harmonics} $\{\psi_{n,k,j}\}_{j}$, i.e., eigenfunctions of the Laplace-Beltrami operator $-\Delta_{\S^{n-1}}$ corresponding to the eigenvalue 
\begin{equation}\label{def: eigenvalues}
\lambda_{n,k}:=k(k+n-2)\,,
\end{equation}
 normalized so that $\|\psi_{n,k,j}\|_{L^2(\S^{n-1})}=1$.  As it is well known,  $d_{n,k}:=\mathrm{dim}H_{n,k}<+\infty$\,, and  we have the orthogonal decomposition
\begin{equation}
\label{eq:decomsphericalharm}
L^2(\S^{n-1};\R^n)=\bigoplus_{k=0}^\infty H_{n,k}\,.
\end{equation} 
In fact,  the spherical harmonics are not just mutually $L^2$-orthogonal, but also $W^{1,2}$-orthogonal. In particular, if  for $u\in W^{1,2}(\S^{n-1};\R^n)$,  we write 
$$u=\sum_{k=0}^\infty \sum_{j=1}^{ d_{n,k}} \alpha_{n,k,j} \psi_{n,k,j}\,, $$
then we have the \textit{Parseval identities}
\begin{equation}
\label{eq:parseval}
\int_{\mb S^{n-1}} |u|^2 = \sum_{k=0}^\infty \sum_{j=1}^{ d_{n,k}} |\alpha_{n,k,j}|^2,
\qquad \int_{\mb S^{n-1}} |\n_T u|^2 = \sum_{k=0}^\infty \sum_{j=1}^{ d_{n,k}} \lambda_{n,k}|\alpha_{n,k,j}|^2\,.
\end{equation}
We consider the orthogonal projection 
\begin{equation}\label{projections_in_spherical_harmonics}
\Pi_{n,k}\colon L^2(\mb S^{n-1};\R^n)\to H_{n,k}\,.
\end{equation}
Clearly $H_{n,0}$ is the span of the constant function  equal to 1, so
\begin{equation}\label{eq:zeroproj}
\Pi_{n,0}(u)=\fint_{\S^{n-1}} u\,,
\end{equation}
while $H_{n,1}$ is the span of the coordinate functions $\{x_i\}_{i=1}^n$, and  it is easy to check that
\begin{equation} \label{eq:firstproj}
\Pi_{n,1}(u)(x) = \nabla u_h(0)x\,,
\end{equation}
where $u_h\colon \overline{B_1}\to \R^n$ denotes the  component-wise harmonic extension of $u\in L^2(\S^{n-1};\R^n)$  in the interior of the unit ball. From \eqref{def: eigenvalues} and \eqref{eq:parseval} we deduce several sharp forms of the Poincar\'e inequality, for instance 
\begin{align}
\label{eq:firstpoincare}
\fint_{\mb S^{n-1}} \big|\n_T u\big|^2 & \geq (n-1) \fint_{\mb S^{n-1}} \big|u- \Pi_{n,0}(u)\big|^2,\\
 \label{eq:secondpoincare}
\fint_{\mb S^{n-1}} \big|\n_T (u- \Pi_{n,1}(u))\big|^2 & \geq 2n \fint_{\mb S^{n-1}} \big|u- \big(\Pi_{n,0}(u)+ \Pi_{n,1}(u)\big)\big|^2 .
\end{align}
Another relatively simple consequence of \eqref{eq:parseval} is the sharp estimate
\begin{equation}
\label{eq:estimateharm}
\fint_{B_1} |\n u_h|^2\,\mathrm{d}x \leq \frac{n}{n-1} \fint_{\mb S^{n-1}} |\n_T u|^2\,\mathrm{d}\H^{n-1}\,,
\end{equation}
whose proof can be found for instance in \cite[Lemma C.2]{zemas2022rigidity}).

\subsection{Conformal maps}\label{subs:conformal_maps}
We now turn to conformal maps of the sphere. We define $\Mob(\mb S^{n-1})$ to be the group of orientation-preserving M\"obius transformations of $\mb S^{n-1}$, that is,
\begin{equation}\label{def: Mobius_group} 
\Mob(\mb S^{n-1}) := \{R \phi_{\xi,\lambda}: R \in \SO(n), \xi \in \S^{n-1}, \lambda>0\}\,,
\end{equation}
where 
\[
\phi_{\xi,\lambda}:=\sigma_{\xi}^{-1}\circ i_{\lambda}\circ\sigma_{\xi}\,,
\] 
$\sigma_{\xi}\colon \S^{n-1}\to T_{\xi}\S^{n-1}\cup\{\infty\}$ being the  stereographic projection  from $-\xi\in \S^{n-1}$\,,  and $i_{\lambda}:T_{\xi}\S^{n-1}\mapsto T_{\xi}\S^{n-1}$ the dilation in $T_{\xi}\S^{n-1}$ by factor $\lambda>0$. 
Explicitly, we have
\begin{gather*}
\sigma_\xi(x)=-\frac{1}{1+\langle x,\xi\rangle}(x- \langle x,\xi\rangle \xi)\,,
\quad
\sigma_\xi^{-1}(y)
=- \Big( \frac{2}{1+|y|^2}\Big) y +\Big(\frac{1-|y|^2}{1+|y|^2}\Big)\xi,
\\
\phi_{\xi,\lambda} (x) = \frac{-\lambda^2 (1-\langle x,\xi\rangle )\xi + 2 \lambda(x-\langle x,\xi\rangle \xi) + (1+\langle x,\xi\rangle) \xi}{\lambda^2 (1-\langle x,\xi\rangle) + (1+ \langle x,\xi\rangle)}\, ,
\end{gather*}
for all $x\in\S^{n-1}$ and $y\in T_\xi\S^{n-1}$.

We next recall  a well-known inequality between the conformally invariant $(n-1)$-Dirichlet energy and the degree of maps on the sphere.  
\begin{lemma}\label{Wente_on_sphere}
For every $u\in W^{1,n-1}(\S^{n-1};\S^{n-1})$ there holds
\begin{equation}\label{ineq: wente_on_sphere}
\fint_{\S^{n-1}}\bigg(\frac{|\nabla_T u|^2}{n-1}\bigg)^{\frac{n-1}{2}}\, \mathrm{d}\mathcal{H}^{n-1} \geq |\deg u|\,,
\end{equation}
with equality if and only if $u$ is weakly conformal.
\end{lemma}
\begin{proof}
Applying the arithmetic mean-geometric mean inequality to the eigenvalues of $\sqrt{\nabla_Tu^t\nabla_T u}$, 
 together with 
  the fact that  $|u|=1$ $\H^{n-1}$-a.e. on $\S^{n-1}$,   from \eqref{def: degree} 
 we obtain
\begin{align}\label{AM_GM_on_the_sphere}
\begin{split}	
\fint_{\S^{n-1}}\bigg(\frac{|\nabla_T u|^2}{n-1}\bigg)^{\frac{n-1}{2}}&\geq \fint_{\S^{n-1}}\sqrt{\mathrm{det}(\nabla_Tu^t\nabla_Tu)}\\
& =\fint_{\S^{n-1}}\bigg|\bigwedge_{i=1}^{n-1}\partial_{\tau_i}u\bigg|
\geq \fint_{\S^{n-1}}\bigg|\big\langle u,\bigwedge_{i=1}^{n-1}\partial_{\tau_i}u\big\rangle\bigg|\geq |\mathrm{deg}(u)|\,. 
\end{split}
\end{align}
If equality holds then we must have equality in the arithmetic mean-geometric mean inequality, which means that all of the eigenvalues of $\sqrt{\nabla_Tu^t\nabla_T u}$ must be the same  $\H^{n-1}$-a.e. on $\mb S^{n-1}$. In other words, if equality holds,  $u$ is weakly conformal.
\end{proof}

As an immediate consequence of Lemma \ref{Wente_on_sphere} and \eqref{admissible_class_of_maps}, we have the following.

\begin{corollary}\label{cor:deficit}
For $u\in \mc A_{\mb S^{n-1}}$, the deficit in \eqref{deficit_on_S_n-1} satisfies $\delta_{n-1}(u)\geq 0$. We have $\delta_{n-1}(u)=0$ if and only if $u$ is weakly conformal, which by Liouville's Theorem holds if and only if $u\in \Mob(\mb S^{n-1})$.
\end{corollary}

\begin{remark} \label{wente_in_R_n} 
\normalfont For maps $u\in (W^{1,n-1}\cap L^\infty)(\S^{n-1};\R^n)$ the quantity
\[V_n(u):=\fint_{\S^{n-1}}\bigg\langle u,\bigwedge_{i=1}^{n-1}\partial_{\tau_i}u\bigg\rangle\,\mathrm{d}\mathcal{H}^{n-1}\]
gives the \textit{relative extrinsic volume} of the image of $u$: for instance,  if $u$ is a smooth embedding, then $V_n(u)$ is the  signed  volume of the open set in $\R^n$ whose boundary is $u(\S^{n-1})$, divided by the volume of the unit ball. Of course, if $u\in W^{1,n-1}(\S^{n-1};\S^{n-1})$, we have $V_n(u) = \deg u$, cf.\ \eqref{def: degree}.  In fact, $V_n(u)$ extends to an analytic functional in $W^{1,n-1}(\S^{n-1};\R^n)$, for which \textit{Wente's isoperimetric inequality} holds:
\begin{equation}\label{eq:isoperimetric}
\fint_{\S^{n-1}}\bigg(\frac{|\nabla_T u|^2}{n-1}\bigg)^{\frac{n-1}{2}} \geq \fint_{\S^{n-1}}\sqrt{\mathrm{det}(\nabla_Tu^t\nabla_Tu)}\geq  \left| \fint_{\S^{n-1}}\bigg\langle u,\bigwedge_{i=1}^{n-1}\partial_{\tau_i}u\bigg\rangle\right|^{\frac{n-1}{n}}\,,
\end{equation}
the first inequality being again a consequence of the arithmetic mean-geometric mean inequality and the  second  
one being the functional form of the isoperimetric inequality,  see e.g.\ \cite{almgren1986iso}, \cite[equation (1.9)]{zemas2022rigidity},   \cite[Lemma~1.3]{muller1990},  or \cite{wente1969}. Equality holds in \eqref{eq:isoperimetric} if and only if the image $u(\mb S^{n-1})$ is another round sphere and, in addition, $u$ is weakly conformal, that is, if and only if $u$ is  a Möbius transformation  of $\S^{n-1}$  up to translation and scaling. 
\end{remark}

Sequences with zero-mean and vanishing deficit exhibit strong compactness properties, see e.g.\  \cite[Lemma A.3]{zemas2022rigidity}:

\begin{lemma}[Compactness]\label{lemma:compactness}
Let $(u_j)_{j\in\N}\subset W^{1,n-1}(\mb S^{n-1}; \mb S^{n-1})$ be a sequence such that 
$$\fint_{\mb S^{n-1}} u_j =0, \qquad \deg u_j = 1, \qquad \lim_{j\to \infty}\delta_{n-1}(u_j) =0 \,.$$
Then there  exists  $R\in \SO(n)$ such that, up to a not-relabeled subsequence, 
$$u_j\to R\,\textup{id}_{\S^{n-1}} \text{ strongly in } W^{1,n-1}(\mb S^{n-1}; \mb S^{n-1})\,.$$
\end{lemma}

\begin{proof}
The condition $\lim_{j\to \infty} \delta_{n-1}(u_j)=0$ can be rewritten as
$$\lim_{j\to \infty} \fint_{\S^{n-1}}\bigg(\frac{|\nabla_T u_j|^2}{n-1}\bigg)^{\frac{n-1}{2}}\, \mathrm{d}\mathcal{H}^{n-1}=1\,.$$
Thus, up to a subsequence we have $u_j \weak u$ weakly in $W^{1,n-1}(\mb S^{n-1};\mb S^{n-1})$; in particular, this convergence is also strong in $L^{n-1}(\S^{n-1};\S^{n-1})$, and so 
\[\Pi_{n,0} (u)=\fint_{\mb S^{n-1}} u = 0\,.\]  
By the sequential weak lower semicontinuity of the conformal Dirichlet energy,  Jensen's inequality,  and the Poincar\'e inequality \eqref{eq:firstpoincare}, we have
\begin{align*}
1 = \lim_{j\to \infty} \fint_{\S^{n-1}}\bigg(\frac{|\nabla_T u_j|^2}{n-1}\bigg)^{\frac{n-1}{2}}& \geq \fint_{\S^{n-1}}\bigg(\frac{|\nabla_T u|^2}{n-1}\bigg)^{\frac{n-1}{2}}\\
&  \geq \bigg(\fint_{\S^{n-1}}\frac{|\nabla_T u|^2}{n-1}\bigg)^{\frac{n-1}{2}} 
\geq \Big(\fint_{\mb S^{n-1}} |u|^2 \Big)^{\frac{n-1}{2}} = 1\,,
\end{align*}
hence equality holds throughout. This already shows the strong convergence $u_j \to u$ in $W^{1,n-1}$, so $\deg u =1$ as well.   Since $u$ satisfies \eqref{eq:firstpoincare} with equality, by \eqref{eq:parseval} we see that $u$ must be linear,  i.e.\ $u(x)=Rx$ for some $R\in \R^{n\times n}$. Finally,  since $u\colon \mb S^{n-1}\to \mb S^{n-1}$ has degree one, it must be that $R\in\SO(n)$. \qedhere
\end{proof}

A basic topological fact is that, up to a M\"obius map, the zero-mean condition of Lemma \ref{lemma:compactness} can always be achieved:

\begin{lemma}\label{lemma:topological}
Given $u\in \mathcal{A}_{\mathbb{S}^{n-1}}$ there exists $\psi\in \Mob(\mathbb{S}^{n-1})$ so that 
\begin{equation}\label{eq: mean_value_0}
\fint_{\S^{n-1}} u\circ\psi=0\,.	
\end{equation}	
\end{lemma}

This lemma is well-known, cf.\  the proof of Liouville's theorem in \cite[Appendix A]{zemas2022rigidity} as well as Step 1 in the proof of \cite[Theorem 1.1]{hirsch2022mobius2} for the case $n=3$. However, to make the presentation reasonably self-contained, we also revise the argument here.

\begin{proof} Assume first that $u\in \mathcal{A}_{\S^{n-1}}$ is smooth,  so in particular $u$ is surjective. If $$b_u:=\fint_{\S^{n-1}} u=0\,,$$ then \eqref{eq: mean_value_0} is trivially satisfied for $\psi:=\mathrm{id}_{\S^{n-1}}$. If $b_u\neq 0$, we claim that there exists $\xi_0 \in \S^{n-1}$ and $\lambda_0>0$ such that \eqref{eq: mean_value_0} holds true for $\psi:=\phi_{\xi_0,\lambda_0}$.
Indeed, consider the map $F:\S^{n-1} \times [0,1]\mapsto \overline {B}_1$ defined as 
\begin{equation*}\label{homotopy_for_fixing_the_center}
F(\xi,\lambda):=\fint_{\S^{n-1}}u\circ\phi_{\xi,\lambda} \ \ \mathrm{for}\ \lambda\in(0,1]\,, \quad \mathrm{and} \ F(\xi,0)
:=u(\xi)=
\lim_{\lambda\searrow 0}F(\xi,\lambda)\,.
\end{equation*}
The map $F$ is continuous with $F(\S^{n-1},0)=u(\S^{n-1})=\S^{n-1}$ and $F(\S^{n-1},1)=\{b_u\}$, hence $F$ is a continuous homotopy between $\S^{n-1}$ and the point $b_u\in \overline {B}_1\diagdown \{0\}$.  Therefore there exists $\lambda_0\in(0,1)$ and $\xi_0\in \S^{n-1}$ such that $F(\xi_0,\lambda_0)=0$, as wished.

In the general case of a map $u\in \mathcal{A}_{\S^{n-1}}$, by the approximation property given in \cite[Lemma 7, Section I.4.]{brezis1995degree}, there exists a sequence $(u_j)_{j\in\N}\subset C^{\infty}(\S^{n-1};\S^{n-1})$ with the property that
\begin{equation*}\label{approximation_property}
u_j\underset{j\to \infty}{\longrightarrow} u \mathrm{\ strongly \ in} \ W^{1,n-1}(\S^{n-1};\S^{n-1}) \mathrm{\ \ and \ \ } \deg u_j=\deg u=1\ \ \forall j\in \mathbb{N}\,.
\end{equation*} 
Up to passing to a not-relabeled subsequence, we can without loss of generality also suppose that $u_j\rightarrow u$ and $\nabla_T u_j\rightarrow \nabla_T u$ pointwise $\mathcal{H}^{n-1}$-a.e. on $\S^{n-1}$, as $j\to\infty$. Since the maps $u_j$ are smooth and surjective, by the previous argument there exist $(\xi_j)_{j\in\N} \subset  \S^{n-1}$ and $(\lambda_j)_{j\in\N}
 \subset (0,1]$ so that for every $j\in \mathbb{N}$,
\begin{equation*}
\fint_{\S^{n-1}} u_j\circ \phi_{\xi_j,\lambda_j}=0\,.
\end{equation*}
Up to subsequences we can suppose further that $\xi_j\rightarrow\xi_0\in \S^{n-1}$ and $\lambda_j\rightarrow\lambda_0\in [0,1]$ as $j\to \infty$, thus $\phi_{\xi_j,\lambda_j}\to \phi_{\xi_0,\lambda_0}$ pointwise $\mathcal{H}^{n-1}$-a.e. on $\S^{n-1}$ and also weakly in $W^{1,n-1}(\S^{n-1};\S^{n-1})$.

In fact $\lambda_0\in (0,1]$, i.e.\ the Möbius transformations $(\phi_{\xi_j,\lambda_j})_{j\in\N}$ do not converge to the trivial map $\phi_{\xi_0,0}(x)\equiv\xi_0$. Indeed, suppose for the sake of contradiction that this were the case. Then $u_j\circ\phi_{\xi_j,\lambda_j}\to u(\xi_0)$ pointwise $\mathcal{H}^{n-1}$-a.e. and $|u_j\circ\phi_{\xi_j,\lambda_j}|\equiv1$, so we could use the Dominated Convergence Theorem to infer that
\begin{align*}
\begin{split}
u(\xi_0)&=\fint_{\S^{n-1}}u(\xi_0)=\lim_{j\to\infty}\fint_{\S^{n-1}} u_j\circ\phi_{\xi_j,\lambda_j}=0\,, \\
|u(\xi_0)|&=\fint_{\S^{n-1}} |u(\xi_0)|=\lim_{j\to\infty}\fint_{\S^{n-1}}|u_j\circ\phi_{\xi_j,\lambda_j}|=1\,,
\end{split}
\end{align*}
which is  a contradiction. 
Having justified that $\lambda_0\in (0,1] $, what we actually obtain by the Dominated Convergence Theorem is that 
\begin{equation*}
\fint_{\S^{n-1}} u\circ \phi_{\xi_0,\lambda_0}= 0\,,
\end{equation*}
which gives \eqref{eq: mean_value_0} in the general case.
\end{proof}

\subsection{Boundary integrals related to null-Lagrangians.}\label{null_Lagrangians}
In this subsection we recall some useful facts about Jacobian subdeterminants  and the boundary expressions of their integrals. Before doing so,    however,   we introduce some notation from linear algebra. 

Let $M\in \R^{n\times n}$ with its set of eigenvalues being labeled as $\{\mu_1,\dots,\mu_n\}$ (be them real or complex).   We then have  
\begin{align}\label{eq:id_sym_poly}
\det(I_n+M)=1+\sum_{k=1}^n \sigma_k(M)\,,
\end{align}
where $\sigma_k(M)$ denotes the $k$-th elementary symmetric polynomial in the eigenvalues of $M$:
\begin{equation}\label{eq:symmetric_polies}
\sigma_k(M):=\sum_{1\leq i_1<\dots<i_k\leq n}\mu_{i_1}\dots\mu_{i_k}\,.
\end{equation}
Note that the $k$-homogeneity of $\sigma_k$ implies the Euler identity
\begin{align}\label{eq:sigmak_hom}
\sigma_k(M)=\frac 1k  \sigma_k'(M):M\,,
\end{align}
where $\sigma_k'(M)\in\R^{n\times n}$ is the gradient of $\sigma_k$ (with respect to the $M$-variable)\,.
With this notation, we have the following.

\begin{lemma}\label{lemma_on_null_Lagrangians}
Let $w\in (W^{1,n-1}\cap L^\infty)(\S^{n-1};\R^n)$ and $W:\overline{B}_1\to \R^n$ be any regular extension such that $W|_{\S^{n-1}}=w$. Then,
\begin{equation}\label{eq:boundary_integrals}
\fint_{B_1} \det(I_n +\nabla W)\,\mathrm{d}x= 1+\sum_{k=1}^n \frac{n}{k}\fint_{\S^{n-1}}\langle w,[\sigma'_k(\nabla_T wP_T^t)]^tx\rangle\,\mathrm{d}\H^{n-1}\,,
\end{equation}
where $\sigma_k$ are as in \eqref{eq:symmetric_polies}. In particular, \eqref{eq:boundary_integrals} holds true for $w_h$, the component-wise harmonic extension of $w$ in $B_1$\,.
\end{lemma}

\begin{proof}
By a standard approximation argument we can without restriction assume for the proof that $u\in C^\infty(\S^{n-1};\R^n)$. 
Since for every $k=1,\dots,n$,  $\sigma_k$ is a linear combination of $k\times k$ minors,   $\sigma_k$    is a null Lagrangian, that is
\begin{align}\label{eq:sigmak_null}
\dv \sigma_k'(\nabla U)=0\quad\text{ for all } U\in C^\infty(B_1;\R^n)\,,
\end{align}
see e.g.\ \cite{Ball1981}.
We now consider a particular smooth extension
 $\overline W\colon \overline B_1\to\R^n$ of $w$
 which is constant in the radial direction.
 Explicitly, one may set
\begin{align*}
\overline W(y)=
\begin{cases}
\eta(y)w(\frac{y}{|y|}) &\text{for }y\neq 0,\\
0 &\text{for }y=0,
\end{cases}
\end{align*}
where $\eta\in C^\infty(\overline B_1;\R_{ +})$ is any smooth cut-off satisfying
$\mathbf 1_{\{|y|\geq 1/2\}} \leq \eta(y) \leq \mathbf 1_{\{|y|\geq 1/4\}}$.
Since the radial derivative of $\overline W$ vanishes on $\mathbb S^{n-1}$, 
 for every $i,j = 1,\dots,n$ we have   
\begin{equation}\label{full_gradient_on_S_n-1}
(\nabla \overline W)_{ij} 
=\sum_{m=1}^{n-1}\langle\nabla_T w^i,\tau_m\rangle\langle e_j,\tau_m\rangle= (\nabla_T w P_T^t)_{ij}
\quad\H^{n-1}\text{-a.e. on }\S^{n-1}.
\end{equation}
%
%
%
%
Using the fact that the Jacobian determinant is a null-Lagrangian, \eqref{eq:id_sym_poly}, \eqref{eq:sigmak_hom}, \eqref{eq:sigmak_null} and \eqref{full_gradient_on_S_n-1}, for every extension $W:\overline{B}_1\to \R^n$ of $w$, we obtain

\begin{align*}
\begin{split}
\fint_{B_1} \det(I_n +\nabla W)\,\mathrm{d}x
&= \fint_{B_1} \det(I_n +\nabla \overline W)\,\mathrm{d}x
\\ &= 1+\sum_{k=1}^n \fint_{B_1}\sigma_k(\nabla \overline W)\,\mathrm{d}x 
\\
&=1+\sum_{k=1}^n \frac{1}{k} \fint_{B_1}\sigma'_k(\nabla \overline W)\colon\nabla \overline W\,\mathrm{d}x
\\
& =1+\sum_{k=1}^n \frac{1}{k} \fint_{B_1}\mathrm{div}(\sigma'_k(\nabla \overline W)\overline W)  \,\mathrm{d}x 
\\
&= 1+\sum_{k=1}^n \frac{n}{k} \fint_{\S^{n-1}}\langle\sigma'_k(\nabla \overline W)\overline W,x\rangle\,\mathrm{d}\H^{n-1}\\
&=1+\sum_{k=1}^n \frac{n}{k}\fint_{\S^{n-1}}\langle\sigma'_k(\nabla_T wP_T^t)w,x\rangle\,\mathrm{d}\H^{n-1}\,,
\end{split}
\end{align*}
  which proves exactly \eqref{eq:boundary_integrals}. \qedhere
\end{proof}

\begin{remark}\label{1st_last_symmetric_polynomials}
\normalfont We observe that $\sigma_1'(M)=I_n \ \forall M\in \R^{n\times n}$, so the term corresponding to $k=1$ in the right hand side of \eqref{eq:boundary_integrals} is equal to $n\fint_{\mathbb S^{n-1}} \langle w, x\rangle\,.$ Similarly, for $k=n$ it is clear that
\begin{equation}\label{last_term_in_the_expansion}
\fint_{\S^{n-1}}\langle w,[\sigma'_n(\nabla_T wP_T^t)]^tx\rangle\,\mathrm{d}\H^{n-1}=\fint_{B_1}\mathrm{det}\nabla W
=\fint_{\S^{n-1}}\bigg\langle w,\bigwedge_{i=1}^{n-1} \partial_{\tau_i}w\bigg\rangle\,\mathrm{d}\H^{n-1}\,,
\end{equation}
while for every $k\in\{2,\dots,n-1\}$, since $\sigma_k'(M)$ is a $(k-1)$-homogeneous polynomial in the entries of $M$, for every $w\in W^{1,n-1}(\S^{n-1};\R^n)$ we have the trivial estimate 
\begin{equation}\label{eq:trivial_estimate_for_intermediate_terms}
\frac{n}{k}\bigg|\fint_{\S^{n-1}}\langle w,[\sigma'_{ k}(\nabla_T wP_T^t)]^tx\rangle\,\mathrm{d}\H^{n-1}\bigg|\leq C_{n,k}
\fint_{\S^{n-1}}|w||\nabla_Tw|^{k-1}\,\mathrm{d}\H^{n-1}\,,
\end{equation}
where $C_{n,k}>0$ is a constant depending only on $n$ and $k$\,.
\end{remark}

\section{Proof of Theorem \ref{main_thm}.}\label{sec:proof}
In this section we prove Theorem \ref{main_thm}. As we noted in the Introduction, only the case $n\geq 4$ is new,  and is the case of study here,   the proof being substantially simpler for $n=3$  (cf. \cite{hirsch2022mobius2}). For the reader's convenience we split the proof in several steps.\\ 

\begin{step}{0. Reduction to maps with small deficit}
We start with the standard observation that, for every $u\in \mathcal{A}_{\S^{n-1}}$  (recall \eqref{admissible_class_of_maps}) , we have
\begin{equation}\label{eq:trivial_estimate_for_large_deficit}
\fint_{\S^{n-1}}|\nabla_Tu-P_T|^{n-1}\lesssim \fint_{\S^{n-1}}|\nabla_Tu|^{n-1}+1\lesssim \delta_{n-1}(u)+1\,,
\end{equation}
and therefore without restriction we may assume that 
\begin{equation}\label{eq:apriori_small_deficit}
0\leq \delta_{n-1}(u)\leq \delta_0
\,.
\end{equation}	
Indeed, if $\delta_{n-1}(u)>\delta_0$, where $\delta_0:=\delta_0(n)>0$ is a small but fixed dimensional constant that will be suitably chosen later, in view of \eqref{eq:trivial_estimate_for_large_deficit} we see that \eqref{conf_S_S_quantitative} trivially holds with $\phi:=\mathrm{id}_{\S^{n-1}}$. \medskip
\end{step}

\begin{step}{1. Reduction to maps with zero mean value} Letting $\psi\in \Mob(\mb S^{n-1})$ be the map provided by Lemma \ref{lemma:topological}, we take $\tilde u:=u\circ\psi$. Since both the deficit and the topological degree are conformally invariant, we see that $\tilde u\in \mathcal{A}_{\S^{n-1}}$ and
\begin{equation}\label{properties_of_tilde_u}
\delta_{n-1}(\tilde u)=\delta_{n-1}(u)\,, \ \ \deg \tilde u=\deg u=1\,, \ \mathrm{ and\ \ } \fint_{\S^{n-1}} \tilde u =0\,.
\end{equation} 
  Hence, if we prove \eqref{conf_S_S_quantitative} for $\tilde u$, then it automatically also holds for $u$.  
\end{step}	\\[2pt]

\begin{step}{2(a). Local estimate in a $W^{1,n-1}$-neighbourhood of the identity}
In this step we prove \eqref{conf_S_S_quantitative}   for $\tilde u$,   under the additional assumption that
\begin{equation}\label{eq:close_to_identity}
\fint_{\S^{n-1}} |\nabla_T \tilde u-P_T|^{n-1}\leq \theta
\,,
\end{equation}	
where $\theta:=\theta(n)>0$ is a sufficiently small but fixed dimensional constant that will be suitably chosen later.

For $R\in \SO(n)$ arbitrary,  consider the map $v_R\in  (W^{1,n-1}\cap L^\infty) (\S^{n-1};\R^n)$ defined via
\begin{equation}\label{def: v_R_map}
v_R(x):=\tilde u(x)-Rx\,.
\end{equation}
In particular, since $  \tilde u $ is $\S^{n-1}$-valued, \eqref{def: v_R_map} implies that for $\H^{n-1}$-a.e.\ $x\in \S^{n-1}$ we have
\[1=|\tilde u(x)|^2=|v_R(x)+Rx|^2=|v_R(x)|^2+2\langle v_R(x),Rx\rangle+1\,, \]
or, equivalently,
\begin{equation}\label{eq: pointwise_identity_for_v_R}
\langle v_R,Rx\rangle=-\frac{1}{2}|v_R|^2\, \quad \H^{n-1}\text{-a.e.\ on } \S^{n-1}\,.
\end{equation}

The next ingredient we need is a pointwise inequality for vectors derived in \cite[Lemma 2.1(ii)]{figalli2022sobolev}.  For the reader's convenience, we repeat its statement here in a weaker form than the original one, that   is more   adequate for our purposes.  

\smallskip

\begin{lemma}\label{lemma:figalli}
Let $m\in \N$, $p\geq 2$. For every $\kappa  \in (0,1) $ there exists a  $c_0:=c_0(p,\kappa) >0 $ so that, for every two vectors $X,Y\in \R^m$, we have
\begin{align}
\label{eq: vector_algebraic_inequality}
\begin{split}
|X+Y|^p & \geq |X|^p+ p|X|^{p-2}\langle X,Y\rangle   +\frac{(1-\kappa)p}{2}|X|^{p-2}|Y|^2+c_0|Y|^p 
\,.
 \end{split}
\end{align}
\end{lemma}

\smallskip
We apply \eqref{eq: vector_algebraic_inequality} with 
\[m:=n(n-1)\,, \quad p:=n-1\geq 2\,, \quad X:=RP_T\,, \quad Y:=\nabla_T v_R\,.\]
 Using \eqref{def: v_R_map} and the fact that $|RP_T|^2=|P_T|^2=n-1$\,, we deduce that $\H^{n-1}$-a.e. on $\S^{n-1}$ we have 
\begin{align*}\label{eq:pointwise_ineq_for_v_R}	
\begin{split}
|\nabla_T \tilde u|^{n-1}&\geq |RP_T|^{n-1}+(n-1)|RP_T|^{n-3}RP_T\colon \nabla_Tv_R+c_0|\nabla_Tv_R|^{n-1}\\
& \quad + \frac{1-\kappa}{2}(n-1)|RP_T|^{n-3}|\nabla_Tv_R|^2 \\
&=(n-1)^{\frac{n-1}{2}}\bigg(1+RP_T\colon\nabla_T v_R +\frac{1-\kappa}{2}|\nabla_T v_R|^2\bigg)+c_0|\nabla_T v_R|^{n-1}\, .
\end{split}
\end{align*}
Integrating the last inequality over $\S^{n-1}$, and after some simple algebraic manipulations (and redefining the value of the constant $c_0>0$), we find
\begin{equation}\label{eq:deficit_ineq_for_v_R}	
\delta_{n-1}(\tilde u)\geq \frac{1-\kappa}{2}\fint_{\S^{n-1}}|\nabla_T v_R|^2+\fint_{\S^{n-1}}RP_T\colon \nabla_Tv_R+c_0\fint_{\S^{n-1}}|\nabla_Tv_R|^{n-1}\,.
\end{equation}
Writing in the ambient space coordinates $v_R=(v_R^1,\dots,v_R^n)$, an integration by parts on $\S^{n-1}$, together with the fact that $\lambda_{n,1}=n-1$  (cf.\ \eqref{def: eigenvalues}),  yields
\begin{align}\label{eq:linear_term_becomes_quadratic}
\begin{split}
\fint_{\S^{n-1}}RP_T\colon \nabla_T v_R&= \sum_{k=1}^n \fint_{\S^{n-1}}\langle \nabla_T(Rx)^k, \nabla_T v_R^k\rangle \\
& =\sum_{k,\ell=1}^n R_{k\ell}\fint_{\S^{n-1}}\langle \nabla_Tx^\ell, \nabla_T v_R^k\rangle\\
&  = \sum_{k,\ell=1}^n R_{k\ell}\fint_{\S^{n-1}}(-\Delta_{\S^{n-1}}x^\ell) v_R^k
\\
&=(n-1)\fint_{\S^{n-1}}\langle Rx, v_R\rangle   =-\frac{n-1}{2}\fint_{\S^{n-1}}|v_R|^2\,,
\end{split}
\end{align}
where in the last equality we used \eqref{eq: pointwise_identity_for_v_R}.  Plugging \eqref{eq:linear_term_becomes_quadratic} in \eqref{eq:deficit_ineq_for_v_R},  we obtain
\begin{equation}\label{eq:new_deficit_ineq_for_v_R}	
\delta_{n-1}(\tilde u)\geq \frac{1-\kappa}{2}\bigg(\fint_{\S^{n-1}}|\nabla_T v_R|^2-\frac{n-1}{1-\kappa}\fint_{\S^{n-1}}|v_R|^2\bigg)+c_0\fint_{\S^{n-1}}|\nabla_Tv_R|^{n-1}\,.
\end{equation}

Noting that $v_R\in W^{1,n-1}(\S^{n-1};\R^n)\subset W^{1,2}(\S^{n-1};\R^n)$, we can consider the decomposition of $v_R$ in spherical harmonics.  Recalling \eqref{projections_in_spherical_harmonics}, by \eqref{properties_of_tilde_u}, the fact that $\fint_{\S^{n-1}}Rx=0$   and \eqref{eq:zeroproj},   we have
\begin{equation}\label{eq:mean_value_of_v_R_zero}
\Pi_{n,0}(v_R) = \fint_{\S^{n-1}}v_R=0\,,
\end{equation}
while \eqref{eq:firstproj} yields
\[\Pi_{n,1} (v_R) =\nabla (v_R)_h(0)x=(\nabla \tilde u_h(0)-R)x\,.\]
By the orthogonality of the decomposition \eqref{eq:decomsphericalharm} in $W^{1,2}(\S^{n-1};\R^n)$, the improved Poincar\' e inequality \eqref{eq:secondpoincare}   and \eqref{eq:mean_value_of_v_R_zero}, we have:
\begin{align}
\label{eq: improved_Poincare}
\fint_{\S^{n-1}}|\nabla_Tv_R|^2&=\fint_{\S^{n-1}}|\nabla_T(v_R-\Pi_{n,1}(v_R))|^2+\fint_{\S^{n-1}}|\nabla_T\Pi_{n,1}(v_R)|^2 \nonumber\\
&\geq 2n\fint_{\S^{n-1}}|v_R-\Pi_{n,1}(v_R)|^2+\fint_{\S^{n-1}}|(\nabla \tilde u_h(0)-R)P_T|^2\\
&= 2n\fint_{\S^{n-1}}|v_R|^2+\fint_{\S^{n-1}}|(\nabla \tilde u_h(0)-R)P_T|^2-2n\fint_{\S^{n-1}}|(\nabla\tilde u_h(0)-R)x|^2 \nonumber \,.
\end{align} 
Note that 
\[\fint_{\S^{n-1}}x_kx_{k'}=\frac{\delta^{kk'}}{n}\,\quad \forall k,k'=1,\dots,n\,;\]
thus, for every  $A\in \R^{n\times n}$,  we can compute
\begin{align}\label{eq:int_A_P_T}
\begin{split}
\fint_{\S^{n-1}}|AP_T|^2&=\sum_{i,j=1}^{n,n-1}\fint_{\S^{n-1}}\bigg(\sum_{k=1}^nA_{ik}(P_T)_{kj}\bigg)^2\\
& =\sum_{i,j=1}^{n,n-1}\sum_{k,k'=1}^{n}A_{ik}A_{ik'}\fint_{\S^{n-1}}\langle e_k,\tau_j\rangle\langle e_{k'},\tau_j\rangle  \\
&=\sum_{i,k,k'=1}^n A_{ik}A_{ik'}\fint_{\S^{n-1}}\bigg\langle \sum_{j=1}^{n-1}\langle e_k,\tau_j\rangle\tau_j,e_{k'} \bigg\rangle\\
& =\sum_{i,k,k'=1}^n A_{ik}A_{ik'}\fint_{\S^{n-1}}\langle e_k-x_kx,e_{k'}\rangle \\
& =\big(1-\tfrac{1}{n}\big)\sum_{i,k,k'=1}^n A_{ik}A_{ik'}\delta^{kk'}=\frac{n-1}{n}|A|^2\,,
\end{split}
\end{align}
and
\begin{align}\label{eq:int_Ax}
\begin{split}
\fint_{\S^{n-1}}|Ax|^2&=\sum_{i=1}^{n}\fint_{\S^{n-1}}\bigg(\sum_{k=1}^n A_{ik}x_k\bigg)^2\\
& =\sum_{i,k,k'=1}^{n}A_{ik}A_{ik'}\fint_{\S^{n-1}}x_kx_{k'}=\sum_{i,k,k'=1}^n A_{ik}A_{ik'}\frac{\delta^{kk'}}{n}=\frac{1}{n}|A|^2\,.
\end{split}
\end{align}
In view of \eqref{eq:int_A_P_T} and \eqref{eq:int_Ax}  (for $A:=\nabla \tilde u_h(0)-R$),   \eqref{eq: improved_Poincare} can be rewritten as
\medskip
\begin{equation}\label{eq:rewritten_improved_Poincare}
\fint_{\S^{n-1}}|\nabla_Tv_R|^2  \geq   2n\fint_{\S^{n-1}}|v_R|^2-\frac{  n+1 }{n}|\nabla \tilde u_h(0)-R|^2\,.
\end{equation}
\medskip
In particular, plugging  \eqref{eq:rewritten_improved_Poincare} in \eqref{eq:new_deficit_ineq_for_v_R}, we get
\smallskip
\begin{align*}
\delta_{n-1}(\tilde u)&\geq \frac{1-\kappa}{2}\left(c_{n,\kappa}\fint_{\S^{n-1}}|v_R|^2-\frac{  n+1 }{n}|\nabla \tilde u_h(0)-R|^2\right)+c_0\fint_{\S^{n-1}}|\nabla_Tv_R|^{n-1}\,,
\end{align*}
\medskip
where $c_{n,\kappa}:=2n-\frac{n-1}{1-\kappa}$. Choosing $\kappa:= \kappa_n= \frac{n+1}{2n} \in (0,1) $ we have $c_{n,\kappa}=0$, and so
$$\delta_{n-1}(\tilde u) \geq -\frac{  (n^2-1)  
}{  4n^2  
}|\nabla \tilde u_h(0)-R|^2+c_0\fint_{\S^{n-1}}|\nabla_Tv_R|^{n-1}\,.
$$
Recalling \eqref{def: v_R_map}, we deduce that
\begin{equation}\label{eq:1_key_estimate_for_stability}
\fint_{\S^{n-1}}|\nabla_T\tilde u-RP_T|^{n-1}\leq c_n \bigg(\delta_{n-1}(\tilde u)+|\nabla\tilde u_h(0)-R|^2\bigg)\,,
\end{equation}
where $c_n >0 $ is a dimensional constant. Recall that, in the above argument, $R\in \SO(n)$ is arbitrary. Therefore, \eqref{eq:1_key_estimate_for_stability} and \eqref{properties_of_tilde_u}   directly   imply \eqref{conf_S_S_quantitative} 
provided that we have
\begin{equation}\label{eq:dist_of_linear_part_from_SO(n)}
\mathrm{dist}^2(\nabla \tilde u_h(0);\SO(n))\lesssim \delta_{n-1}(\tilde u)\,.
\end{equation} 
  Indeed, if this is the case, then we could apply \eqref{eq:1_key_estimate_for_stability} for $R\in \SO(n)$ such that 
$$|\nabla\tilde u_h(0)-R|=\mathrm{dist}(\nabla \tilde u_h(0);\SO(n))\,,$$
and obtain \eqref{conf_S_S_quantitative} with $\phi =R\psi^{-1}\in \Mob(\S^{n-1})$. The verification of \eqref{eq:dist_of_linear_part_from_SO(n)} will be provided in the next (sub)step.  
\end{step}\\[2pt]

\begin{step}{2(b). Verification of \eqref{eq:dist_of_linear_part_from_SO(n)}} For brevity, let us  here  set 
\begin{equation}\label{eq: A_matrix}
A:=\nabla\tilde u_h(0)\,.
\end{equation}
By the mean-value property   of harmonic functions,   \eqref{eq:estimateharm},  \eqref{eq:close_to_identity} and Jensen's inequality, 
\begin{align}\label{nabla_at_0_almost_id}
\begin{split}	
|A-I_n|^2&=\left|\fint_{B_1} \nabla \tilde u_h-I_n\right|^2 \leq \fint_{B_1} |\nabla \tilde u_h-I_n|^2\\
& \leq\frac{n}{n-1}\fint_{\S^{n-1}}|\nabla_T\tilde u-P_T|^2\\
&\leq \frac{n}{n-1}\bigg(\fint_{\S^{n-1}}|\nabla_T\tilde u-P_T|^{n-1}\bigg)^{\frac{2}{n-1}}\lesssim \theta^{\frac{2}{n-1}}\ll 1\,,
\end{split}
\end{align}
and by choosing $\theta\in(0,1)$ sufficiently small, we can take $A$ to be invertible and such that 
\begin{equation}\label{matrixinversedeterminantestimates}
|A|^2, |A^{-1}|^2\ \in [n-1,n+1]\ \ \ \mathrm{and \ \ \ } \mathrm{det}A\in \left[\tfrac{1}{2},\tfrac{3}{2}\right]\,.
\end{equation}
By the polar decomposition, $A=R_0 \,   U_A $ with $R_0\in \SO(n)$ and $  U_A:=\sqrt{A^tA} $ symmetric positive-definite.  If we label the eigenvalues of $  U_A $ as $0<\alpha_1\leq\dots\leq \alpha_n$ and set 
\begin{equation}\label{def: lambdas}
\lambda_i:=\alpha_i-1\,, \qquad \Lambda^2:=\sum_{i=1}^n\lambda_i^2\quad  (\Lambda>0) \,,
\end{equation}
 we have
\begin{equation}\label{L2_small_conformal_s_2}
\Lambda^2 =|U_A-I_n|^2=|A-R_0|^2=\mathrm{dist}^2\big(A;\SO(n)\big)\leq|A-I_n|^2\lesssim\theta^{\frac{2}{n-1}}\ll 1\,.
\end{equation}
Next we define $ w(x):=A^{-1}(\tilde u(x)-Ax)$, so that
\begin{align}\label{eq:tildeu_w}
\tilde u(x)= A (x + w(x))\,.
\end{align}
Note that \eqref{eq:1_key_estimate_for_stability} and \eqref{matrixinversedeterminantestimates} imply
\begin{align}\label{eq:nabla_w_Ln-1}
\fint_{\S^{n-1}}|\nabla_T w|^{n-1}\lesssim \fint_{\mb S^{n-1}} |\n_T \tilde u -   R_0  P_T|^{n-1} +  | R_0  -A|^{n-1}    \lesssim   
\delta_{n-1}(\tilde u)+\Lambda^2
\,.
\end{align}
Moreover, dropping the last nonnegative term in \eqref{eq:new_deficit_ineq_for_v_R} and letting $\kappa\to 0$ therein, we have
\begin{align*}
2\delta_{n-1}(\tilde u)
&\geq 
\fint_{\S^{n-1}}|\nabla_T v_R|^2 -(n-1)\fint_{\S^{n-1}}|v_R|^2 \\
& =\fint_{\S^{n-1}}|\nabla_T(v_R-\Pi_{n,1} ( v_R ))|^2-(n-1)\fint_{\S^{n-1}}
|v_R-\Pi_{n,1} ( v_R )|^2 \\
&\geq \left(1-\frac{n-1}{2n}\right)\fint_{\S^{n-1}}|\nabla_T(v_R-\Pi_{n,1} ( v_R ) )|^2,
\end{align*}
where the last inequality follows again from \eqref{eq:secondpoincare}.
In terms of $w$, by  \eqref{matrixinversedeterminantestimates}, and the fact that $\Pi_{n,1}( \tilde u )  (x)=Ax$ (cf.\ \eqref{eq:firstproj}   and \eqref{eq: A_matrix} ) and  $\Pi_{n,1}(R_0x)=R_0x$,  the last inequality   implies
\begin{align}\label{eq:nabla_w_L2}
\fint_{\S^{n-1}}|\nabla_T w|^{2}\leq c_{n} \delta_{n-1}(\tilde u)\,.
\end{align}

Next, we wish to expand the identity
\begin{align}
\label{eq:expanddeg}
1=\deg(\tilde u)=\fint_{B_1}\det(\nabla \tilde u_h)\,\mathrm{d}x=\det A\fint_{B_1} \det(I_n +\nabla w_h)\,\mathrm{d}x\,,
\end{align}
in terms of $w$.   By Lemma 
\ref{lemma_on_null_Lagrangians}, we obtain    
\begin{align}
\label{eq:chaindeg}
\begin{split}
1
&
= \det A\fint_{B_1} \det(I_n +\nabla w_h)=
\det A 
\left( 1+\sum_{k=1}^n \frac{n}{k}\fint_{\S^{n-1}}\langle \sigma'_k(\nabla_T w  P_T^t )w,x \rangle
 \right).
 \end{split}
\end{align}
  By Remark \ref{1st_last_symmetric_polynomials}, the term corresponding to $k=1$ in the right hand side of \eqref{eq:chaindeg}     equals  
\begin{align*}
n\fint_{\mathbb S^{n-1}} \langle w, x\rangle =n\fint_{\mathbb S^{n-1}} \langle (\tilde u(x)-Ax),  A^{-t}  x \rangle =0\,,
\end{align*}
  where we used   again   \eqref{eq: A_matrix} and \eqref{eq:firstproj}\,.   
Hence \eqref{eq:chaindeg} becomes
\begin{equation}\label{eq: det_expansion_1}
1
=
\det A 
+\det A \sum_{k=2}^n \frac nk \fint_{\S^{n-1}}\langle \sigma'_k(\nabla_T w  P_T^t )w,x \rangle\,.
\end{equation}
Recalling the notation   $U_A=\sqrt{A^tA}$, setting for notational convenience
\begin{equation}\label{def:sigma_k_lambda}
\sigma_k(\lambda):=\sigma_k(U_A-I_n)\, \quad \forall k=1,\dots,n\,,
\end{equation}
 and using   also the algebraic identity
   $$\det A =\det  U_A= \prod_{k=1}^n (1+\lambda_k)=1+\sum_{k=1}^n \sigma_k(  \lambda 
)\,,$$ 
 (cf.\   \eqref{def: lambdas}   and \eqref{eq:id_sym_poly}),   identity \eqref{eq: det_expansion_1}   can be rewritten as   
\begin{align*}
0
&=\sum_{k=1}^n\sigma_k(  \lambda 
)  +\det A \sum_{k=2}^n \frac nk \fint_{\S^{n-1}}\langle\sigma'_k(\nabla_T w  P_T^t )w,x \rangle\,.
\end{align*}
Recalling   that, by \eqref{def: lambdas},  $\Lambda^2=\sum_{k=1}^n\lambda_k^2 =\sigma_1(  \lambda 
)^2 -2\sigma_2(  \lambda 
)$, we deduce
\begin{align}\label{eq:Lambda_identity}
\frac 12\Lambda^2
=  \sigma_1(\lambda) +\frac{\sigma_1(\lambda)^2}{2}   
+\sum_{k=3}^n \sigma_k(  \lambda 
)
+\det A \sum_{k=2}^n \frac nk \fint_{\S^{n-1}}\langle\sigma'_k(\nabla_T w  P_T^t )w,x  \rangle\,.
\end{align}

Next we estimate all terms in the right-hand side of   \eqref{eq:Lambda_identity}.
First note that
\begin{equation}\label{lambda_identity}
  \sigma_1(\lambda) +\frac{\sigma_1(\lambda)^2}{2}  
=\frac{|A|^2-n}{2}\,.
\end{equation}
By the mean value property of harmonic functions, Jensen's inequality and \eqref{eq:estimateharm}, we have
\begin{align*}
|A|^2 & =|\nabla \tilde u_h(0)|^2
=\left| \fint_{B_1} \nabla \tilde u_h \right|^2
\leq \fint_{B_1}|\nabla \tilde u_h|^2 \leq \frac{n}{n-1}\fint_{\S^{n-1}}|\nabla_T \tilde u|^2,
\end{align*}
and again by Jensen's inequality,
\begin{align*}
|A|^2\leq n\left[\fint_{\S^{n-1}}\left(\frac{|\nabla_T \tilde u|^2}{n-1}\right)^{\frac{n-1}{2}}\right]^{\frac{2}{n-1}}
=n\bigg(1+\delta_{n-1}(\tilde u)\bigg)^{\frac{2}{n-1}}
\leq n +\frac{2n}{n-1}\delta_{n-1}(\tilde u)\,,
\end{align*}
  where we used the inequality $(1+t)^\alpha\leq 1+\alpha t$\,, which holds true for every $t\geq 0$ and 
$0\leq \alpha\leq 1$. 
 In particular, \eqref{lambda_identity} and the above inequality yield %
\begin{align}\label{eq:sigma1+sigma1^2/2}
\sigma_1(\lambda)+\frac{\sigma_1(\lambda)^2}{2}   
=\frac{|A|^2-n}{2}
\leq \frac{n}{n-1}\delta_{n-1}(\tilde u)\,,
\end{align}
  which takes care of the first two terms in the right-hand side of \eqref{eq:Lambda_identity}.
Since $\sigma_k$ is $k$-homogeneous and $\Lambda$ is small by \eqref{L2_small_conformal_s_2}, the third term in the right-hand side of \eqref{eq:Lambda_identity} is estimated by  (recall \eqref{def:sigma_k_lambda}),  
\begin{align}\label{eq:sum_sigmak_3}
\sum_{k=3}^n\sigma_k(  \lambda 
)  \lesssim |U_A-I|^3=\Lambda^3   \lesssim \theta^{\frac{1}{n-1}}
\Lambda^2\,.
\end{align}
For the last sum in \eqref{eq:Lambda_identity}, we consider first all terms but the last. Using that $\sigma_k'$ is $(k-1)$-homogeneous,   by \eqref{eq:trivial_estimate_for_intermediate_terms}    
we have
\begin{align}\label{eq:interpolation_1}
\begin{split}
\sum_{k=2}^{n-1}\frac{n}{k}\fint_{\S^{n-1}}\langle\sigma_k'(\nabla_Tw  P_T^t )  w,x \rangle
&
  \lesssim  
\sum_{k=2}^{n-1}\fint_{\S^{n-1}}|\nabla_T w|^{k-1}|w|
\\
&
  \lesssim  \fint_{\S^{n-1}}|w|\left(|\nabla_T w|+ |\nabla_T w|^{n-2}\right)\,,
\end{split}
\end{align}
  where in the last inequality we used the fact that 
\[1\leq k-1\leq n-2 \implies |Z|^{k-1}\leq |Z|+|Z|^{n-2} \quad \forall Z\in \R^m\,.\] 
 
  Note that $|w| |\nabla_T w|\leq |w|^2+|\nabla_Tw|^2$ and  
  $ \|w\|_{L^\infty(\S^{n-1})} \lesssim 1 
$   by \eqref{matrixinversedeterminantestimates} and \eqref{eq:tildeu_w}.   Moreover, since we are here considering the case $n\geq 4$,   for any $\eps>0$ there holds
\[ |\nabla_T w|^{n-2}\leq   \eps^{4-n}   
|\nabla_T w|^2 +\eps |\nabla_T w|^{n-1}\,,\]
 as can be checked by distinguishing the cases $|\nabla_T w|\leq \frac{1}{\eps}$ and $|\nabla_T w|\geq \frac{1}{\eps}$. Combining these pointwise inequalities with \eqref{eq:interpolation_1} results in    the estimate  
\begin{align*}
\sum_{k=2}^{n-1}\frac{n}{k}\fint_{\S^{n-1}}\langle\sigma_k'(\nabla_Tw  P_T^t )  w,x  \rangle
&
  \lesssim  
 \fint_{\S^{n-1}}|w|^2
 +C(\eps) \fint_{\S^{n-1}}|\nabla_T w|^2
 +\eps \fint_{\S^{n-1}}  |\nabla_T w|^{n-1}
\,,
\end{align*}
  where $C(\eps):=1+\eps^{4-n}>0$.  
Combining Poincar\'e's inequality \eqref{eq:firstpoincare}  for the zero-mean map $w$ and the estimates \eqref{eq:nabla_w_L2} and \eqref{eq:nabla_w_Ln-1} on the $L^2$ and $L^{n-1}$ norms of $\nabla_T w$, this implies
\begin{align}\label{eq:sum_sigmakDw_low}
\sum_{k=2}^{n-1}\frac{n}{k}\fint_{\S^{n-1}}\langle\sigma_k'(\nabla_Tw  P_T^t )w,x \rangle
&
\leq C(n,\eps) \delta_{n-1}(\tilde u) +\eps \Lambda^2,
\end{align}
for any $\eps>0$ and some $C(n,\eps)>0$.
It remains to estimate the very last   summand in the last   term in \eqref{eq:Lambda_identity}.
For this we invoke first   \eqref{last_term_in_the_expansion} and   Wente's isoperimetric inequality, cf.\ \eqref{eq:isoperimetric}, which gives
\begin{align*}
\fint_{\S^{n-1}}\langle \sigma_{n}'( \nabla_T wP_T^t )  w,x  \rangle
=\fint_{\S^{n-1}}\bigg\langle w,\bigwedge_{i=1}^{n-1}\partial_{\tau_i}w\bigg\rangle
\leq \left[ \fint_{\S^{n-1}}\left(\frac{|\nabla_T w|^2}{n-1}\right)^{\frac{n-1}{2}}\right]^{\frac{n}{n-1}}.
\end{align*}
Using the $L^{n-1}$ estimate \eqref{eq:nabla_w_Ln-1} for $\nabla_T w$,   the convexity of the function $t\mapsto t^{n-1}$ ($t>0$), \eqref{eq:apriori_small_deficit} and \eqref{L2_small_conformal_s_2},   we deduce
\begin{align}\label{eq:last_term_Lambda}
\fint_{\S^{n-1}}\langle \sigma_{n}'( \nabla_T wP_T^t )  w,x  \rangle
&
  \lesssim   
\delta_{n-1}(\tilde u)^{\frac{n}{n-1}}+\Lambda^{\frac{2n}{n-1}}
 \lesssim 
 \delta_{n-1}(\tilde u) +  \theta^{\frac{2}{(n-1)^2}} \Lambda^2\,,
\end{align}
  with $\theta\in (0,1)$ as in \eqref{eq:close_to_identity}.   
Combining \eqref{eq:sigma1+sigma1^2/2}, \eqref{eq:sum_sigmak_3}, \eqref{eq:sum_sigmakDw_low} and \eqref{eq:last_term_Lambda} to estimate the right-hand side of \eqref{eq:Lambda_identity} we  obtain,
  \[\frac{1}{2}\Lambda^2\lesssim C(n,\eps)\delta_{n-1}(\tilde u)+ \big(\theta^{\frac{2}{(n-1)^2}}+\eps\big)\Lambda^2\,,\] 
  and by choosing $\theta\in (0,1)$ and $\e\in(0,1)$ small enough, we obtain
\begin{align*}
\frac 12 \Lambda^2 \leq C(n,\e) \delta_{n-1}(\tilde u) +\frac 14 \Lambda^2,
\end{align*}
Absorbing the last term in the left-hand side and recalling that $\Lambda=\dist(\nabla\tilde u_h(0);\SO(n))$,
this implies \eqref{eq:dist_of_linear_part_from_SO(n)} and concludes Step 2(b).

\medskip	
\end{step}

\begin{step}{3. (From a local to a global estimate via $W^{1,n-1}$-compactness)} Arguing by contradiction, suppose that the statement of Theorem \ref{main_thm} is false. Then, for every $k\in \mathbb{N}$ there exists a map $u_k\in \mathcal{A}_{\S^{n-1}}$ with $\delta_{n-1}(u_k)>0$ such that 

\begin{equation}\label{eq:contradictory_estimate}
\fint_{\S^{n-1}} |\nabla_T u_k-\nabla_T\phi|^{n-1}\geq k\delta_{n-1}(u_k)\, \quad \text{for all } \phi\in \Mob(\S^{n-1})\,.
\end{equation}
In particular, for $\phi:=\mathrm{id}_{\S^{n-1}}\in \Mob(\S^{n-1})$, by the convexity of the function $t\mapsto t^{n-1}$ $(t>0)$ we have 
\begin{align*}
\begin{split}
k\delta_{n-1}(u_k)& \leq \fint_{\S^{n-1}} |\nabla_T u_k-P_T|^{n-1}\\
&\leq 2^{n-2}\fint_{\S^{n-1}} \left(|\nabla_T u_k|^{n-1}+|P_T|^{n-1}\right)
= 2^{n-2}(n-1)^{\frac{n-1}{2}} [\delta_{n-1}(u_k)+ 2]\,,
\end{split}
\end{align*}
which, for  $k> \beta_n:=2^{n-2}(n-1)^{\frac{n-1}{2}}$, can be rewritten as $\delta_{n-1}(u_k)\leq \frac{2\beta_n}{k-\beta_n}\,.$ By letting $k\to\infty$ we obtain 
\[\lim_{k\to\infty}\delta_{n-1}(u_k)=0\,.\]
We can now use the compactness result from Lemma \ref{lemma:compactness} to obtain a contradiction: by Lemma \ref{lemma:topological}, we find $\psi_k\in \Mob(\S^{n-1})$ and $R\in \SO(n)$ so that the $v_k:=u_k\circ\psi_k\in \mathcal{A}_{\S^{n-1}}$ satisfy
\begin{gather}
\label{eq:zeroaverage} \fint_{\S^{n-1}} v_k=0\,.
\end{gather} 
Hence, by Lemma \ref{lemma:compactness}, we have  $v_k\to R\, \mathrm{id}_{\S^{n-1}}$  strongly in  $W^{1,n-1}(\S^{n-1};\S^{n-1})$ as $j\to \infty$, up to a not-relabeled subsequence.
Without loss of generality (up to considering $R^{-1} v_k$ instead of $v_k$ if necessary) we can also suppose that $R=I_n$. Then, for the constant $\theta \in (0,1) $ chosen in Step 2, we can find $k_0:=k_0(\theta)\in \mathbb{N}$ such that 
\begin{equation*}
\fint_{\S^{n-1}} |\nabla_T v_k-P_T|^{n-1}\leq \theta\ll 1 \quad \text{for all }  k\geq k_0\,.
\end{equation*}
The sequence $(v_k)_{k\geq k_0}$  satisfies both \eqref{eq:zeroaverage} and the  assumption \eqref{eq:close_to_identity} of Step 2: hence, by this step, we deduce that there exist $(\phi_k)_{k\geq k_0}\subset \Mob(\S^{n-1})$ such that
\begin{equation*}\label{eq:final_contradiction}
\fint_{\S^{n-1}} |\nabla_T u_k-\nabla_T (\phi_k\circ\psi_k^{-1})|^{n-1} =  \fint_{\S^{n-1}} |\nabla_T v_k-\nabla_T \phi_k|^{n-1}\leq  C_n  \delta_{n-1}(v_k) \quad \text{for all } k\geq k_0\,,
\end{equation*}
which clearly contradicts \eqref{eq:contradictory_estimate}.
\end{step}
	
\appendix

\section{On the optimality of Theorem \ref{main_thm}}\label{optimality_example}

We explain here Remark \ref{sharpness}, by considering a detailed 
example in the same spirit as the second example in \cite[Remark 1.2]{figalli2022sobolev}. For the sake of making the calculations simpler, we equivalently consider maps in the class
\begin{equation}\label{eq: new_adm_class}
\hspace{-0.5em}\mathcal{A}_{\R^{n-1}}:=\bigg\{u\in \dot{W}^{1,n-1}(\R^{n-1};\S^{n-1})\,,\ \frac{1}{n\omega_n}\int_{\R^{n-1}}\big\langle u,\bigwedge_{i=1}^{n-1}\partial_{x_i} u\big\rangle=1 \bigg\} \,,
\end{equation}
where $(\partial_{x_i})_{i=1,\dots,n-1}$ denote the standard partial derivatives in $\R^{n-1}$ and 
\[\dot{W}^{1,n-1}(\R^{n-1};\S^{n-1}):=\bigg\{v\in W^{1,n-1}_{\mathrm{loc}}(\R^{n-1};\S^{n-1})\colon \int_{\R^{n-1}}|\nabla v|^{n-1}\,,\mathrm{d}x<+\infty\bigg\}\]
is the corresponding homogeneous Sobolev space. The above class of maps can be identified with the class $\mathcal{A}_{\S^{n-1}}$, defined in \eqref{admissible_class_of_maps}, via the inverse stereographic projection through the south pole $-e_n$, given by the map $\phi:\R^{n-1}\cup\{\infty\}\to \S^{n-1}$ defined by
\begin{equation}\label{eq: inverse_stereo}
\phi(x)=\bigg(-\frac{2x_1}{1+|x|^2},\dots,-\frac{2x_{n-1}}{1+|x|^2}, \frac{1-|x|^2}{1+|x|^2}\bigg)\,.
\end{equation}
We note that $\phi$ is orientation-preserving and that
\begin{equation}\label{eq: energy_of_phi}
\int_{\R^{n-1}}|\nabla \phi|^{n-1}\,\mathrm{d}x=\gamma_n:=(n-1)^{\frac{n-1}{2}}n\omega_n\,.
\end{equation}
Analogously,  the group $\Mob(\mb S^{n-1})$ defined in \eqref{def: Mobius_group} can be identified with 
\begin{equation}\label{eq: Mobius_in_R_n-1}
\varPsi:= \{R\phi(\rho(\cdot-x_0))\colon \ R\in \SO(n)\,,\ x_0\in \R^{n-1}\,, \ \rho>0\}\,,
\end{equation}
 since clearly $\varPsi\phi^{-1}$ is a Lie subgroup of $\Mob(\mb S^{n-1})$ with the same dimension (as can be easily checked by considering their corresponding Lie algebras).

With these identifications, \eqref{conf_S_S_quantitative} is equivalent to showing that for every $u\in \mathcal{A}_{\R^{n-1}}$,
\begin{equation}\label{eq: main_estimate_in_R_n-1}
\inf_{\psi \in \varPsi} \int_{\R^{n-1}} \left|\nabla u-\nabla \psi\right|^{n-1} \mathrm{d}x \leq C_n \bigg(\int_{\R^{n-1}}|\nabla u|^{n-1}\,\mathrm{d}x-\gamma_n\bigg)\,,
\end{equation}
where $\gamma_n$ is as in \eqref{eq: energy_of_phi}, for a possibly different dimensional constant $C_n>0$.

Let now $\zeta$ be a non-trivial cut-off function such that
\begin{equation}\label{eq: properties_of_zeta}
\zeta\in C^\infty_c(D_1;\R_+)\,, \quad \zeta|_{D_{1/2}}\equiv 1\,,\quad 0\leq \zeta\leq 1\,,
\end{equation}
where for $\rho>0$ we denote 
$$D_\rho:=\{x\in \R^{n-1}\colon |x|\leq \rho\}\,.$$
For a sequence of numbers $(\eps_{k})_{k\in \N}\subset \R_+$ with $\eps_k\searrow 0$ and points $(x_k)_{k\in\N}\subset \R^{n-1}$ with $|x_k|\to \infty$,  consider the maps $(u_k)_{k\in \N}\subset \mathcal{A}_{\R^{n-1}}$, defined via
\begin{equation}\label{eq: maps_for_optimality}
u_k(x):=\frac{\phi(x)+\eps_k\zeta(x-x_k)e_1}{|\phi(x)+\eps_k\zeta(x-x_k)e_1|}\,,
\end{equation}
for which by the definition of $\zeta$ and the fact that $|\phi|\equiv 1$, we have that
\begin{equation}\label{eq:optimality_maps_almost_identity}
u_k\equiv \phi\ \ \text{in }\ \R^{n-1}\setminus D_1(x_k)\,.
\end{equation}

The sequence $(u_k)_{k\in \N}$ gives the optimality claimed in Remark \ref{sharpness}:

\begin{proposition}
Choosing the sequences $x_k\nearrow \infty$ and $\e_k\searrow 0$ appropriately,  for all $k  \in \N $ large enough, we have
\begin{align}\label{eq:scaling_of_deficit}
\int_{\R^{n-1}}|\nabla u_k|^{n-1}\,\mathrm{d}x-\gamma_n\sim \eps_k^{n-1}+o(\eps_k^{n-1})\,,\\
\label{eq:scaling_of_dist}
\inf_{\psi \in \varPsi} \int_{\R^{n-1}} \left|\nabla u_k-\nabla \psi\right|^{n-1} \mathrm{d}x \sim\eps_k^{n-1}+o(\eps_k^{n-1})\,.
\end{align}	
\end{proposition}

\begin{proof}
Before proceeding with the main part of the proof, we begin with some preliminary computations. Consider  the maps $\hat v:\R^n\setminus\{0\}\to \S^{n-1}$ and $w_k:\R^{n-1}\to \R^n$ defined by 
\begin{equation}\label{eq:auxiliary_maps}
\hat v(y):=\frac{y}{|y|}\quad  \text{ and } \quad w_k(x):=\phi(x)+\eps_k\zeta(x-x_k)e_1\,.
\end{equation} 
Since 
\[\nabla \hat v(y)=\frac{I_n}{|y|}-\frac{y\otimes y}{|y|^3}\,,\]
for every $x\in D_1(x_k):=x_k+D_1$ we can compute
\begin{align}\label{eq: new_gradient}
\begin{split}
\nabla u_k(x)& =\nabla (\hat v\circ w_k)(x)\\
&=\bigg(\frac{I_n}{|w_k(x)|}-\frac{w_k(x)\otimes w_k(x)}{|w_k(x)|^3}\bigg)\nabla w_k(x)\\[2pt]
&=\frac{\nabla \phi(x)}{|w_k(x)|}+\frac{\eps_ke_1\otimes\nabla \zeta(x-x_k)}{|w_k(x)|}-\frac{w_k(x)\otimes w_k(x)}{|w_k(x)|^3}\nabla w_k(x)\,.
\end{split}
\end{align}
Note that 
\begin{align}\label{eq: w_k_aux}
\begin{split}	
|w_k(x)|&=\sqrt{1+2\eps_k\phi^1(x)\zeta(x-x_k)+\eps_k^2\zeta^2(x-x_k)}\\
&=1+\eps_k\phi^1(x)\zeta(x-x_k)+\eps_k^2r_{k,0}(x)\,,
\end{split}
\end{align}
for some remainder term $r_{k,0}:\R^{n-1}\to \R$ with $\|r_{k,0}\|_{L^\infty(\R^{n-1})}\lesssim 1$ (since $|\phi| \equiv  1$ and $0\leq \zeta\leq 1$). Actually, since 
\[ \ \phi^j(x)\to 0 \text{ for all } j=1,\dots, n-1\,\ \mathrm{and }\ \  \nabla\phi(x)\to 0\,, \ \text{as } |x|\to \infty\,,\] 
the points $(x_k)_{k\in \N}\subset \R^{n-1}$ can be chosen such that 
\begin{equation}\label{eq: points_x_k_of_decay}
\|\nabla \phi\|_{L^\infty(D_1(x_k))}+\max_{j\in\{1,\dots,n-1\}}\|\phi^j\|_{L^\infty(D_1(x_k))}\lesssim \eps^2_k\,,
\end{equation}
so that \eqref{eq: w_k_aux} yields
\begin{align}\label{eq: 2_w_k_aux}
\begin{split}	
|w_k(x)|&=1+\eps_k^2r_{k,1}(x)\,,
\end{split}
\end{align}
for some new  $r_{k,1}:\R^{n-1}\to \R$ with $r_{k,1}= 0$ in $\R^{n-1}\setminus D_1(x_k)$ and $\|r_{k,1}\|_{L^\infty(\R^{n-1})}\lesssim 1$.  
Regarding the last term in the last line of \eqref{eq: new_gradient}, using the fact that 
\[(\phi\otimes \phi)\nabla \phi=0\,,\]
since $|\phi| \equiv  1$, 
we can calculate
\begin{align}\label{eq:cubic_term_in_w_k}
\begin{split}
(w_k\otimes w_k)\nabla w_k& = \eps_k\phi^1\big(\phi\otimes\nabla \zeta(\cdot-x_k)\big)+ \eps_k \zeta(\cdot-x_k)\big(\phi\otimes \nabla\phi^1\big)\\ 
&\quad  +\eps_k^2\zeta^2(\cdot-x_k)e_1\otimes \nabla \phi^1 \\
&\quad + \eps_k^2 \zeta(\cdot-x_k)\left(\phi\otimes\nabla \zeta(\cdot-x_k)+\phi^1e_1\otimes\nabla\zeta(\cdot-x_k)\right) \\
&\quad +\eps_k^3\zeta(\cdot-x_k)\big(e_1\otimes\nabla\zeta(x-x_k)\big)\\
& =  \eps_k^2Z_{k,0}(x)\,,
\end{split}
\end{align}
for some $Z_{k,0}:\R^{n-1}\to \R^{n\times (n-1)}$ with $Z_{k,0}  \equiv 0 $ in $\R^{n-1}\setminus D_1(x_k)$ and $\|Z_{k,0}\|_{L^\infty(\R^{n-1})}\lesssim 1$. In the last line of \eqref{eq:cubic_term_in_w_k} we used \eqref{eq: properties_of_zeta} and \eqref{eq: points_x_k_of_decay}.
Therefore, by \eqref{eq: new_gradient}, \eqref{eq: 2_w_k_aux}, \eqref{eq:cubic_term_in_w_k} and another Taylor expansion we deduce that 
\begin{equation}\label{eq: nabla_u_k_final_expression}
\nabla u_k(x)=\nabla \phi(x)+\eps_k\big(e_1\otimes\nabla \zeta(x-x_k)\big)+\eps_k^2Z_{k,1}(x)\,,
\end{equation}
for a map $Z_{k,1}:\R^{n-1}\to\R^{n\times (n-1)}$ with $\|Z_{k,1}\|_{L^\infty(\R^{n-1})}\lesssim 1$. Again, in view of \eqref{eq: properties_of_zeta} and \eqref{eq:optimality_maps_almost_identity}, we can take $Z_{k,1}  \equiv  0$ in $\R^{n-1}\setminus D_1(x_k)$. 

We are now ready to prove our assertions. Note that, by Theorem \ref{main_thm}, it suffices to prove the upper bound in \eqref{eq:scaling_of_deficit} and the lower bound in \eqref{eq:scaling_of_dist}. 

We begin by proving \eqref{eq:scaling_of_deficit}.  
From \eqref{eq: nabla_u_k_final_expression} we have
\begin{align*}
|\nabla u_k(x)|^{n-1}
&
\leq 
\big(|\nabla \phi(x)|+\eps_k |\nabla \zeta(x-x_k)|+\eps_k^2|Z_{k,1}(x)|\big)^{n-1}.
\end{align*}
Expanding the right-hand side and using \eqref{eq: points_x_k_of_decay} to estimate the mixed terms, we deduce the pointwise estimate
\begin{equation}\label{eq:almost_orthogonality_pointwise}
|\nabla u_k|^{n-1}\leq |\nabla \phi|^{n-1}+\eps_k^{n-1}|\nabla\zeta(\cdot-x_k)|^{n-1}+\eps_k^{n}r_{k,1}\,,
\end{equation}
for a new function $r_{k,1}:\R^{n-1}\to[0,+\infty)$ with $r_{k,1}|_{\R^{n-1}\setminus D_1(x_k)}= 0$ and $\|r_{k,1}\|_{L^\infty(\R^{n-1})}\lesssim 1\,.$
Thus,  by \eqref{eq: energy_of_phi} 
we get	
\begin{align*}
\int_{\R^{n-1}} |\n u_k|^{n-1}\,\d x  -\gamma_n & = \int_{\R^{n-1}}\left( |\n u_k|^{n-1} - |\n \phi|^{n-1} \right)\d x  \\
& \leq \e_k^{n-1} \int_{D_1} |\n \zeta|^{n-1} + o(\e_k^{n-1})\,,
\end{align*}
proving the upper bound in \eqref{eq:scaling_of_deficit}. 

Instead of proving \eqref{eq:scaling_of_dist} directly, we first prove it with $\phi$ in place of $\psi$. 
Indeed, by  \eqref{eq:optimality_maps_almost_identity} and \eqref{eq: nabla_u_k_final_expression},
\begin{align*}
\begin{split}
\int_{\R^{n-1}} \left|\nabla u_k-\nabla \phi\right|^{n-1}=\eps_k^{n-1}\int_{D_1(x_k)}|e_1\otimes\nabla \zeta(\cdot-x_k)+\eps_kZ_{k,1}|^{n-1}\,
\end{split}
\end{align*}
and,  by   a Taylor expansion,  
 we deduce that   for $k\in \N$ large enough,  
\begin{align}
\label{eq:estimatewithid}
\begin{split}
 \int_{\R^{n-1}} \left|\nabla u_k-\nabla \phi\right|^{n-1} \geq \e_k^{n-1} \int_{D_1}|\nabla \zeta|^{n-1}+ o(\e_k^{n-1})\,. 
\end{split}
\end{align}
At this point, it remains to verify that
the infimum in \eqref{eq:scaling_of_dist}
is of the same order as  the value of the integral 
 at $\psi:=\phi$.
 To see this, let $ (\psi_k)_{k\in \N}\subset \varPsi$ be such that
\begin{equation}
\label{eq:choicepsik}
\int_{\R^{n-1}} |\n u_k - \n \psi_k|^{n-1}\, \d x =\inf_{ \psi \in \Psi} \int_{\R^{n-1}} |\n u_k - \n \psi|^{n-1}\, \dx + o(\e_k^{n-1})\,,
\end{equation}
and note that we must have $\psi_k \to \phi$ for instance in $\dot W^{1,n-1}(\R^{n-1})$; however, as $\Psi$ is   a   finite dimensional   Lie group , in fact $\psi_k\to \phi$ in any norm.    In particular we have $\psi_k\to\phi$ in $C^1(\R^{n-1})$, 
which implies
\begin{align*}
\int_{D_1(x_k)}|\nabla \psi_k -\nabla\phi|^{n-1}\, \mathrm{d}x =o(\e_k^{n-1})\, ,
\end{align*}
since $ \|\nabla\phi\|_{L^{\infty}(D_1(x_k))} \lesssim \e_k^2 $ by \eqref{eq: points_x_k_of_decay}. Therefore,
\begin{align*}
\int_{\R^{n-1}}|\nabla u_k-\nabla\psi_k|^{n-1}\, \mathrm{d}x
&
\geq \int_{D_1(x_k)}|\nabla u_k -\nabla\psi_k|^{n-1}\,	\d x
\\
&
\geq \int_{D_1(x_k)}|\nabla u_k-\nabla \phi|^{n-1}\, \d x  - o(\e_k^{n-1}) 
\\
&
=\int_{\R^{n-1}}|\nabla u_k-\nabla\phi|^{n-1}\,\d x - o(\e_k^{n-1})\,.
\end{align*}
The last equality is valid because $u_k=\phi$ outside $D_1(x_k)$.
The  lower bound in \eqref{eq:scaling_of_dist} now follows immediately from \eqref{eq:estimatewithid} and \eqref{eq:choicepsik} , and the proof is complete.  
\end{proof}	
	
\section*{Acknowledgements} 
AG was supported by Dr.\ Max R\"ossler, the Walter H\"afner Foundation and the ETH Z\"urich Foundation.
KZ was supported by the Deutsche Forschungsgemeinschaft (DFG, German Research Foundation) under Germany's Excellence Strategy EXC 2044 -390685587, Mathematics M\"unster: Dynamics--Geometry--Structure.  XL was supported by 
the ANR project ANR-22-CE40-0006.

The authors would like to thank the Hausdorff Institute for Mathematics (HIM) in Bonn,
funded by the Deutsche Forschungsgemeinschaft (DFG, German Research Foundation) under Germany's Excellence Strategy – EXC-2047/1 – 390685813,  and the organizers of the Trimester Program ``Mathematics for Complex Materials'' (03/01/2023-14/04/2023, HIM, Bonn) for their hospitality during the period that this work was initiated.

\bibliographystyle{abbrv}
\bibliography{stability-conf-maps}

\begin{thebibliography}{10}

\bibitem{almgren1986iso}
F.~Almgren.
\newblock Optimal isoperimetric inequalities.
\newblock {\em Indiana Univ. Math. J.}, 35:451--547, 1986.

\bibitem{Ball1981}
J.~Ball, J.~Currie, and P.~Olver.
\newblock {Null Lagrangians, weak continuity, and variational problems of
  arbitrary order}.
\newblock {\em J. Funct. Anal.}, 41(2):135--174, 1981.

\bibitem{bernand2019quantitative}
A.~Bernand-Mantel, C.~B. Muratov, and T.~M. Simon.
\newblock A quantitative description of skyrmions in ultrathin ferromagnetic
  films and rigidity of degree {{\(\pm 1\)}} harmonic maps from
  {{\(\mathbb{R}^2\)}} to {{\(\mathbb{S}^2\)}}.
\newblock {\em Arch. Ration. Mech. Anal.}, 239(1):219--299, 2021.

\bibitem{bianchi1991egnell}
G.~Bianchi and H.~Egnell.
\newblock A note on the {Sobolev} inequality.
\newblock {\em J. Funct. Anal.}, 100(1):18--24, 1991.

\bibitem{brezis1995degree}
H.~Br{\'e}zis and L.~Nirenberg.
\newblock Degree theory of {BMO}. {I}: {Compact} manifolds without boundaries.
\newblock {\em Sel. Math., New Ser.}, 1(2):197--263, 1995.

\bibitem{Deng_et_al_2}
B.~Deng, L.~Sun, and J.~Wei.
\newblock Quantitative stability of harmonic maps from {$\mathbb R^2$} to
  {$\mathbb S^2$}.
\newblock {\em arXiv:2111.07630}, 2021.

\bibitem{Deng_et_al_1}
B.~Deng, L.~Sun, and J.~Wei.
\newblock Non-degeneracy and quantitative stability of half-harmonic maps from
  {{\(\mathbb{R}\)}} to {{\(\mathbb{S}\)}}.
\newblock {\em Adv. Math.}, 420:42, 2023.
\newblock Id/No 108979.

\bibitem{Engelstein_2022_Neumayer_Spolaor}
M.~Engelstein, R.~Neumayer, and L.~Spolaor.
\newblock Quantitative stability for minimizing {Yamabe} metrics.
\newblock {\em Trans. Am. Math. Soc., Ser. B}, 9:395--414, 2022.

\bibitem{faraco2005geometric}
D.~Faraco and X.~Zhong.
\newblock Geometric rigidity of conformal matrices.
\newblock {\em Ann. Sc. Norm. Super. Pisa, Cl. Sci. (5)}, 4(4):557--585, 2005.

\bibitem{figalli2010neumayer}
A.~Figalli and R.~Neumayer.
\newblock Gradient stability for the {Sobolev} inequality: the case {{\(p\geq
  2\)}}.
\newblock {\em J. Eur. Math. Soc. (JEMS)}, 21(2):319--354, 2019.

\bibitem{figalli2022sobolev}
A.~Figalli and Y.~R.-Y. Zhang.
\newblock Sharp gradient stability for the {Sobolev} inequality.
\newblock {\em Duke Math. J.}, 171(12):2407--2459, 2022.

\bibitem{frank}
R.~L. Frank.
\newblock Degenerate stability of some {Sobolev} inequalities.
\newblock {\em Ann. Inst. Henri Poincar{\'e}, Anal. Non Lin{\'e}aire},
  39(6):1459--1484, 2022.

\bibitem{friesecke2002theorem}
G.~Friesecke, R.~D. James, and S.~M{\"u}ller.
\newblock A theorem on geometric rigidity and the derivation of nonlinear plate
  theory from three-dimensional elasticity.
\newblock {\em Commun. Pure Appl. Math.}, 55(11):1461--1506, 2002.

\bibitem{groemer}
H.~Groemer.
\newblock {\em Geometric Applications of Fourier Series and Spherical
  Harmonics}.
\newblock Cambridge University Press, 1996.

\bibitem{hirsch2022mobius2}
J.~Hirsch and K.~Zemas.
\newblock A note on a rigidity estimate for degree {$\pm 1$} conformal maps on
  {$\mathbb S^2$}.
\newblock {\em Bull. Lond. Math. Soc.}, 54(1):256--263, 2022.

\bibitem{hungerbuhler97}
N.~Hungerb{\"u}hler.
\newblock {{\(m\)}}-harmonic flow.
\newblock {\em Ann. Sc. Norm. Super. Pisa, Cl. Sci., IV. Ser.}, 24(4):593--631,
  1997.

\bibitem{Iwaniec2002}
T.~Iwaniec and G.~Martin.
\newblock {\em Geometric function theory and nonlinear analysis}.
\newblock Oxford Math. Monogr. Oxford: Oxford University Press, 2001.

\bibitem{Lemaire}
L.~Lemaire.
\newblock Applications harmoniques de surfaces riemanniennes.
\newblock {\em J. Differ. Geom.}, 13:51--78, 1978.

\bibitem{LinWang}
F.~Lin and C.~Wang.
\newblock {\em The analysis of harmonic maps and their heat flows.}
\newblock Hackensack, NJ: World Scientific, 2008.

\bibitem{zemas2022rigidity}
S.~Luckhaus and K.~Zemas.
\newblock Rigidity estimates for isometric and conformal maps from
  {{\(\mathbb{S}^{n-1}\)}} to {{\(\mathbb{R}^n\)}}.
\newblock {\em Invent. Math.}, 230(1):375--461, 2022.

\bibitem{muller1990}
S.~M{\"u}ller.
\newblock Higher integrability of determinants and weak convergence in {{\(L^
  1\)}}.
\newblock {\em J. Reine Angew. Math.}, 412:20--34, 1990.

\bibitem{mullersverakyan}
S.~M{\"u}ller, V.~{\v{S}}ver{\'a}k, and B.~Yan.
\newblock Sharp stability results for almost conformal maps in even dimensions.
\newblock {\em J. Geom. Anal.}, 9(4):671--681, 1999.

\bibitem{reshetnyak1994stability}
Y.~G. Reshetnyak.
\newblock {\em Stability theorems in geometry and analysis. {Translated} from
  the {Russian} by {N}. {S}. {Dairbekov} and {V}. {N}. {Dyatlov}. {Revised} and
  updated translation}, volume 304 of {\em Math. Appl., Dordr.}
\newblock Dordrecht: Kluwer Academic Publishers, rev. and updated transl.
  edition, 1994.

\bibitem{rupflin}
M.~Rupflin.
\newblock Sharp quantitative rigidity results for maps from {$\mathbb S^2$} to
  {$\mathbb S^2$} of general degree.
\newblock {\em arXiv:2305.17045}, 2023.

\bibitem{topping2020}
P.~M. Topping.
\newblock A rigidity estimate for maps from {$S^2$} to {$S^2$} via the harmonic
  map flow.
\newblock {\em Bull. Lond. Math. Soc.}, 55(1):338--343, 2023.

\bibitem{wente1969}
H.~C. Wente.
\newblock An existence theorem for surfaces of constant mean curvature.
\newblock {\em J. Math. Anal. Appl.}, 26:318--344, 1969.

\bibitem{wood}
J.~C. Wood.
\newblock {\em Harmonic mappings between surfaces}.
\newblock PhD thesis, Warwick University, 1974.

\bibitem{Yan_1}
B.~Yan.
\newblock Remarks on {{\(W^{1,p}\)}}-stability of the conformal set in higher
  dimensions.
\newblock {\em Ann. Inst. Henri Poincar{\'e}, Anal. Non Lin{\'e}aire},
  13(6):691--705, 1996.

\bibitem{Yan_2}
B.~Yan and Z.~Zhou.
\newblock Stability of weakly almost conformal mappings.
\newblock {\em Proc. Am. Math. Soc.}, 126(2):481--489, 1998.

\end{thebibliography}

\end{document}